\documentclass[12pt]{amsart}
\usepackage{hyperref,cleveref,amssymb,mathtools,color}

\theoremstyle{plain}
\newtheorem{theorem}{Theorem}[section]
\newtheorem{proposition}[theorem]{Proposition}
\newtheorem{corollary}[theorem]{Corollary}
\newtheorem{lemma}[theorem]{Lemma}

\theoremstyle{definition}
\newtheorem{remark}[theorem]{Remark}

\newtheorem{example}[theorem]{Example}
\newtheorem{definition}[theorem]{Definition}

\newcommand{\abs}[1]{\lvert#1\rvert}
\newcommand{\norm}[1]{\lVert#1\rVert}
\newcommand{\bigabs}[1]{\bigl\lvert#1\bigr\rvert}

\newcommand{\term}[1]{{\textit{\textbf{#1}}}}   
\renewcommand{\mid}{\::\:}

\newcommand{\goestau}{\xrightarrow{\tau}}	
\newcommand{\goeseta}{\xrightarrow{\eta}}	
\newcommand{\goeso}{\xrightarrow{\mathrm{o}}}	
\newcommand{\goesm}{\xrightarrow{\mathrm{m}}}
\newcommand{\goesc}{\xrightarrow{\mathrm{c}}}
\newcommand{\goesu}{\xrightarrow{\mathrm{u}}} 
\newcommand{\goesws}{\xrightarrow{\mathrm{w}^*}} 
\newcommand{\goesucc}{\xrightarrow{\mathrm{ucc}}} 

\DeclareSymbolFont{bbold}{U}{bbold}{m}{n}
\DeclareSymbolFontAlphabet{\mathbbold}{bbold}
\def\one{\mathbbold{1}}

\DeclareMathOperator{\Range}{Range}

\renewcommand{\le}{\leqslant}
\renewcommand{\ge}{\geqslant}

\begin{document}

\title[Net convergence structures]
{Net convergence structures\\ with applications to vector lattices}

\author{M. O'Brien}
\email{mjbeauli@ualberta.ca}

\author{V.G. Troitsky}
\email{troitsky@ualberta.ca}

\address{Department of Mathematical and Statistical Sciences,
  University of Alberta, Edmonton, AB, T6G\,2G1, Canada.}

\author{J.H. van der Walt}
\email{janharm.vanderwalt@up.ac.za}
\address{Department of Mathematics and Applied Mathematics, University of Pretoria,Corner of Lynnwood Road and Roper Street, Hatfield 0083, Pretoria, South
Africa}

\thanks{The second author was supported by an NSERC grant.}

\date{\today}

\begin{abstract}
  Convergence is a fundamental topic in analysis that is most commonly
  modelled using topology. However, there are many natural convergences
  that are not given by any topology; e.g., convergence almost
  everywhere of a sequence of measurable functions and order
  convergence of nets in vector lattices. The theory of convergence
  structures provides a framework for studying more general modes of
  convergence. It also has one particularly striking feature: it is
  formalized using the language of filters. This paper develops a
  general theory of convergence in terms of nets. We show that it is
  equivalent to the filter-based theory and present some translations
  between the two areas. In particular, we provide a characterization
  of pretopological convergence structures in terms of nets. We also
  use our results to unify certain topics in vector lattices with
  general convergence theory.
\end{abstract}

\maketitle

\section{Introduction}
Convergence is an important part of the toolkit for
anyone working in functional analysis, where many deep properties are
often expressed in terms of convergent nets and sequences. While
convergence is often associated with topology, it has long been known
that there are important convergences that are not topological: it was
shown in~\cite{Ordman:66} that convergence almost everywhere is not
given by a topology. Furthermore, there are several important non-topological
convergences in the theory of vector lattices, including order
convergence, uo-convergence, and relative uniform convergence.

The theory of convergence structures was developed in order to handle
non-topological convergences. For an overview of this subject area we
refer the reader to the monograph~\cite{Beattie:02}. Of particular
note, this theory uses the language of filters to describe
convergence; we will refer to this approach as the theory of
\emph{filter convergence structures}. In this paper, we develop a
theory of \emph{net convergence structures} that is equivalent to
filter convergence theory. As the name suggests, our approach will
focus on describing convergence using the language of nets. While
several authors have approached this idea, including
\cite[p.~73]{Kelley:55},
\cite{Katetov:67,Aarnes:72,Arstein:77,Pearson:88},
\cite[p.~168-170]{Schechter:97}, and~\cite{Ho:10,Aydin}, our paper has
several advantages. First and foremost, our theory is equivalent to
the theory of filter convergence structures despite the use of
nets. As a consequence, all results from the theory of filter
convergence structures remain applicable to net convergence
structures. Our definition also handles several set-theoretic
subtleties, like how the collection of all nets in a given set does
not form a set.

This paper can be broken into two main parts: the general theory of
net convergence structures, which is developed in
Sections~\ref{sec:nets}--\ref{sec:cc}, and applications to vector
lattices in Sections~\ref{sec:oconv}
and~\ref{sec:uconv-otop}. Section~\ref{sec:nets} covers preliminaries
and set-theoretic considerations, while the definition of net
convergence structures and their basic properties are presented in
Section~\ref{sec:ncs}. We provide a brief overview of filter
convergence structures in Section~\ref{sec:fcs}, and show they are
equivalent to our theory of net convergence in
Section~\ref{sec:equiv}. Section~\ref{sec:terms} is a dictionary
between several basic concepts from the filter theory and the language
of nets. A reader not interested in filter convergence structures may
skip Sections~\ref{sec:fcs}, \ref{sec:equiv}, and~\ref{sec:terms}: in
a way, the entire point of this paper is that the filter theory is not
required to study convergence structures. We characterize topological
and pretopological convergence structures in terms of nets in
Section~\ref{sec:mix}. In Section~\ref{sec:cc} we reinvestigate
continuous convergence from the perspective of nets. In the final two
sections we discuss applications of our theory to vector lattices:
Section~\ref{sec:oconv} focuses on order convergence from the point of
view of convergence structures, while we identify the Mackey and
topological modifications of order convergence in
Section~\ref{sec:uconv-otop}.

For results concerning topological vector
spaces, we refer the reader to \cite{Kelley:76} and Section~3.1
in~\cite{Aliprantis:06}.

\section{What exactly is a net?}
\label{sec:nets}

The definition of a net is not consistent throughout the literature:
some authors require the index set to be ordered, while others allow
pre-ordered index sets. Some authors may require index sets to be
infinite and without terminal nodes. We will take the most general
approach.

It is a common in the literature to quantify over all nets in a
given set. For example, the standard definition of an order closed set
in a vector lattice goes as follows: a set $C$ is said to be order
closed if for every net $(x_\alpha)$ in~$C$, if the net converges in
order to~$x$, it implies $x\in C$. This definition is ``illegal'' from
the point of view of ZFC: we quantify over all nets in~$C$, hence we
implicitly quantify over all directed sets, which is not
a set. One way to resolve this issue is to use NBG rather than ZFC and
consider classes of nets; this is the approach taken
in~\cite[p.~73]{Kelley:55}. We describe another way to circumvent this
issue while staying within ZFC.

Let $A$ be a set equipped with a relation ``$\le$''. A relation is a
\term{pre-order} if it is reflexive and
transitive. Note that $x\sim y$ if $x\le y$ and $y\le x$ then defines
an equivalence relation. A pre-order may be easily converted into an
order when necessary; this allows one to view any pre-ordered set as
an ordered set. One way to do this is to replace each equivalence
class with a single point to produce a ``quotient order''. Another
approach is to ``straighten up'' each equivalence class by using the
Well Ordering Principle to redefine the relation on each equivalence
class to make it into a totally ordered set.

A set $A$ is \term{directed} if it is pre-ordered
and every two elements have a common descendant; more formally
$$\forall \alpha,\beta\in A\quad\exists \gamma\in
A\mbox{ such that }\alpha\le \gamma\mbox{ and }\beta\le \gamma.$$
Throughout this section $X$ will denote an arbitrary set.  A function
from a directed set $A$ to $X$ is referred to as a \term{net} in~$X$
indexed by~$A$. In this case, instead of writing $x\colon A\to X$ we
write $(x_\alpha)_{\alpha\in A}$ or just $(x_\alpha)$ if the index set
is clear from the context. A \term{sequence} is simply a net indexed
by~$\mathbb N$.

Let $(x_\alpha)_{\alpha\in A}$ be a net in~$X$. We write
$\{x_\alpha\}_{\alpha\in A}$ for the set of all terms of the net:
$\{x_\alpha\mid\alpha\in A\}$, which is nothing but the range of $x$
viewed as a function. If $\alpha_0\in A$ is fixed and we put $A_0=\{\alpha\in
A\mid\alpha\ge\alpha_0\}$ then $A_0$ is again a directed set
under the pre-order induced from~$A$. The restriction of the function
$x$ to $A_0$ is called a \term{tail} of $(x_\alpha)$, and it is
denoted by $(x_\alpha)_{\alpha\ge\alpha_0}$. The range of this new net
will be denoted by $\{x_\alpha\}_{\alpha\ge\alpha_0}$ and called a
\term{tail set} of $(x_\alpha)_{\alpha\in A}$. By
$\langle x_\alpha\rangle_{\alpha\in A}$ we mean the set of all tail sets
of $(x_\alpha)_{\alpha\in A}$; that is,
\begin{displaymath}
  \langle x_\alpha\rangle_{\alpha\in A}
  =\bigl\{\{x_\alpha\}_{\alpha\ge\alpha_0}\mid \alpha_0\in A\bigr\}.
\end{displaymath}

Convergence is often thought of as a ``tail property'', meaning that
altering terms at the ``head'' of a convergent net should not affect
convergence. So, intuitively, convergence should be determined by tail
sets. This idea is embedded into the following two definitions.
Following \cite{Katetov:67}\footnote{also see~\cite{Aarnes:72}}, 
given two nets $(x_\alpha)_{\alpha\in A}$ and $(y_\beta)_{\beta\in B}$ we say
that $(y_\beta)_{\beta\in B}$ is a \term{quasi-subnet} of
$(x_\alpha)_{\alpha\in A}$ and write
\begin{math}
  (y_\beta)_{\beta\in B}\preceq(x_\alpha)_{\alpha\in A}
\end{math}
if for every $\alpha_0\in A$ there exists $\beta_0\in B$ such that
\begin{math}
  \{y_\beta\}_{\beta\ge\beta_0}\subseteq\{x_\alpha\}_{\alpha\ge\alpha_0};
\end{math}
this just means every tail of $(x_\alpha)_{\alpha\in A}$ contains a tail of
$(y_\beta)_{\beta\in B}$ as a \emph{subset}. 

If both
\begin{math}
  (y_\beta)_{\beta\in B}\preceq(x_\alpha)_{\alpha\in A}
\end{math}
and
\begin{math}
  (x_\alpha)_{\alpha\in A}\preceq(y_\beta)_{\beta\in B},
\end{math}
we say that $(x_\alpha)_{\alpha\in A}$ and $(y_\beta)_{\beta\in B}$
are \term{tail equivalent} and write 
\begin{math}
  (x_\alpha)_{\alpha\in A}\sim(y_\beta)_{\beta\in B}.
\end{math}
While this relation will be sufficient for our main results, we
frequently observe a stronger property.  $(x_\alpha)_{\alpha\in A}$
and $(y_\beta)_{\beta\in B}$ are said to be \term{stongly tail
  equivalent}, written
\begin{math}
(x_\alpha)_{\alpha\in A}\approx(y_\beta)_{\beta\in B},
\end{math}
if
\begin{math}
\langle x_\alpha\rangle_{\alpha\in A}=\langle y_\beta\rangle_{\beta\in B}
\end{math}.

It is easy to see that ``$\preceq$'' is reflexive and transitive,
while ``$\sim$'' and ``$\approx$'' are reflexive, transitive, and
symmetric. While it is tempting to say ``$\preceq$'' is a pre-order
and ``$\sim$'' and ``$\approx$'' are equivalence relations on the nets
on~$X$, the latter collection is not a set. The following technical
trick will allow us to work around this issue.

Our idea is to restrict the collection of index sets to a {\it set}
that is large enough to represent any possible tail behaviour
in~$X$. We write $\mathcal P(X)$ for the power set of~$X$ and define
$V_1(X)=X$ and $V_{n+1}(X)=V_n(X)\cup\mathcal P\bigl(V_n(X)\bigr)$
for $n\in\mathbb N$; this is formally known as the {\it
  superstructure} over~$X$. A net $(x_\alpha)_{\alpha\in A}$ in $X$ is
called \term{admissible} if $A\in\bigcup_{n=1}^\infty V_n(X)$.

%
%

Unlike all nets in~$X$, the admissible nets clearly form a set; we
denote it by $\mathfrak N(X)$. It is easy to see that the restriction
of ``$\preceq$'' to $\mathfrak N(X)$ is a pre-order relation, and the
restrictions of ``$\sim$'' and ``$\approx$'' are equivalence
relations.

We will delay the proof of this next theorem
until Section~\ref{sec:equiv} where it will be deduced as an
easy corollary.
\begin{theorem}\label{admis}
  Every net in $X$ is strongly tail equivalent to an admissible net.
\end{theorem}

Since any reasonable definition of convergence for nets should be
stable under tail equivalence, this theorem effectively says we may,
without loss of generality, restrict our attention to admissible
nets. In view of this, we will often identify an arbitrary net
$(x_\alpha)_{\alpha\in A}$ in $X$ with the equivalence class in
$\mathfrak N(X)/\!\!\sim$ consisting of the admissible nets that are
tail equivalent to $(x_\alpha)_{\alpha\in A}$.

\medskip

There are several non-equivalent definitions of the term {\it subnet}
in the literature and, for historical reasons, we would like to
mention the two that are most common. Keep in mind that, in the end,
the exact definition of subnet will be of little concern to us because
our theory will be based on quasi-subnets. Let
$(x_\alpha)_{\alpha\in A}$ be a net, $B$ be a directed set, and
$\varphi\colon B\to A$ such that $\Range\varphi$ is \emph{co-final}
in~$A$, meaning for every $\alpha_0\in A$ there is $\beta_0\in B$ with
$\varphi(\beta)\ge\alpha_0$ whenever $\beta\ge\beta_0$. The
composition $x\circ\varphi\colon B\to X$ is a net in $X$ indexed
by~$B$ and is called a \term{subnet} of $(x_\alpha)_{\alpha\in A}$
according to Kelley in~\cite{Kelley:55}. Willard in~\cite{Willard:70}
requires the additional assumption that $\varphi$ be monotone. Clearly
every Willard-subnet is a Kelley-subnet. We refer the reader to
\cite[pp.~162-165]{Schechter:97} for a thorough discussion about
several alternative definitions of a subnet and relationships between
these definitions. The following result is essentially contained in
Corollary~7.19 in~\cite{Schechter:97}.

\begin{proposition}\label{subnet-q-subnet}
  Every subnet is a quasi-subnet. Every quasi-subnet is
  tail-equivalent to a subnet.
\end{proposition}

\begin{proof}
  It is easy to see that every subnet (either Willard or Kelley) is a
  quasi-subnet. Let $(y_\beta)_{\beta\in B}$ be a quasi-subnet of
  $(x_\alpha)_{\alpha\in A}$. Put
  \begin{math}
    C=\bigl\{(\alpha,\beta)\in A\times B\mid x_\alpha=y_\beta\bigr\}.
  \end{math}
  For $(\alpha,\beta)\in C$, put $\varphi(\alpha,\beta)=\alpha$.  It
  is straightforward that $C$ is directed under the product order
  induced from $A\times B$, and $x\circ\varphi$ is a subnet of
  $(x_\alpha)_{\alpha\in A}$ (in the sense of both Willard and
  Kelley), which is tail equivalent to $(y_\beta)_{\beta\in B}$.
\end{proof}

\begin{example}\label{ex:altern}
  \begin{enumerate}
  \item The sequences $(1,0,0,0,\dots)$ and $(-1,0,0,0,\dots)$ are
    tail equivalent, but not stronlgy tail equivalent. In particular,
    each is a quasi-subnet of the other.  Notice that neither is a
    subnet of the other in the sense of either Kelley or Willard.
  \item 
  Let $(x_n)_{n\in\mathbb N}$ and $(y_n)_{n\in\mathbb N}$ be sequences
  in $\mathbb N$ given by
  \begin{align*}
  &(x_n)_{n\in\mathbb N}=(1,2,3,4,5,6,7,8,\dots)\\
  \text{and} \quad&(y_n)_{n\in\mathbb N}=(2,1,4,3,6,5,8,7,\dots).
  \end{align*}
  It is easy to see that these two sequences are tail equivalent.
  Note that $(x_n)_{n\in\mathbb N}$ is increasing while
  $(y_n)_{n\in\mathbb N}$ is not. So the property of being monotone is
  not preserved under tail equivalence. Also note that the two
  sequences are not strongly tail equivalent. However, each is a
  Kelley-subnet of the other via $\varphi(2n-1)=2n$ and
  $\varphi(2n)=2n-1$ ($n\in\mathbb N)$.

\item The sequences $(-1,0, 1,2,3,\dots)$ and $(0,-1,0,1,2,3,\dots)$
  are strongly tail equivalent, yet one is increasing while the other
  is not: so monotonicity is not preserved under strong tail
  equivalence.
  \end{enumerate}
\end{example}

A net $(x_\alpha)$ in a partially ordered set is said to be increasing
if $x_\alpha\le x_\beta$ in $X$ whenever $\alpha\le\beta$ in~$A$. One
defines decreasing nets similarly.  The preceding example shows that
tail equivalence may ruin important properties of nets, like
monotonicity. This has the potential to cause serious problems for
applications to vector lattices. There is, however, a fix for this
issue.

\begin{proposition}\label{inc-admis}
  Let $X$ be a partially ordered set, and let $(x_\alpha)_{\alpha\in\Lambda}$
  be an increasing net in~$X$. Then there exists an admissible
  increasing net in $X$ which is strongly tail equivalent to
  $(x_\alpha)_{\alpha\in\Lambda}$. A similar result holds for decreasing nets.
\end{proposition}

\begin{proof}
  For each $\alpha\in\Lambda$ let
  $T_\alpha=\{x_\beta\}_{\beta\ge\alpha}$ be the $\alpha$-th tail set of
  the original net. It is clear that $x_\alpha$ is the (unique) least
  element of~$T_\alpha$. Let
  \begin{displaymath}
    \mathcal B=\langle x_\alpha\rangle_{\alpha\in\Lambda}
    =\{T_\alpha\mid\alpha\in\Lambda\}
  \end{displaymath}
  and order it by reverse
  inclusion. This makes
  $\mathcal B$ into a partially ordered directed set. Define a net
  $(y_A)_{A\in\mathcal B}$ by letting $y_A$ be the least element
  of~$A$. It can be easily verified that this net is increasing, and it is
  admissible because $\mathcal B\in\mathcal P\bigl(\mathcal P(X)\bigr)$.

  Let $\alpha_0\in\Lambda$. If $A\ge T_{\alpha_0}$ in $\mathcal B$
  then $y_A\in A\subseteq T_{\alpha_0}$, which implies
  $\{y_A\}_{A\ge T_{\alpha_0}}\subseteq T_{\alpha_0}$. Furthermore,
  for every $\alpha\ge\alpha_0$ in~$\Lambda$ we have
  $T_\alpha\ge T_{\alpha_0}$ in $\mathcal B$ and
  $x_\alpha=y_{T_\alpha}$; hence
  $T_\alpha\subseteq\{y_A\}_{A\ge T_{\alpha_0}}$. This yields
  $\{y_A\}_{A\ge T_{\alpha_0}}=T_{\alpha_0}$, and thus the two nets have
  the same tail sets. Therefore, they are strongly tail equivalent.
  
  The case of decreasing nets may be handled in a similar fashion.
\end{proof}

This proposition motivates the following definition. Let $X$ be a
partially ordered set. An equivalence class in
$\mathfrak N(X)/\!\!\sim$ is said to be \term{increasing} if it contains an
increasing net and \term{decreasing} if it contains a decreasing net.

\begin{remark}\label{same-ind}
  When dealing with a pair of nets, the following trick allows
  one to assume the two nets have the same index set. Let
  $(x_\alpha)_{\alpha\in A}$ and $(y_\beta)_{\beta\in B}$ be two
  nets. The Cartesian product $A\times B$ can be made into a directed
  set by defining $(\alpha_1,\beta_1)\le(\alpha_2,\beta_2)$ whenever
  $\alpha_1\le\alpha_2$ in $A$ and $\beta_1\le\beta_2$ in~$B$. Now define
  two new nets indexed by $A\times B$ via
  \begin{displaymath}
    x'_{(\alpha,\beta)}=x_\alpha\quad\mbox{and}\quad
    y'_{(\alpha,\beta)}=y_\beta,\quad\mbox{where}\quad
    (\alpha,\beta)\in A\times B.
  \end{displaymath}
  It is easy to see that $(x_\alpha)_{\alpha\in A}$ is strongy tail
  equivalent to $(x'_{(\alpha,\beta)})_{(\alpha,\beta)\in A\times B}$
  and  $(y_\beta)_{\beta\in B}$ is strongly tail
  equivalent to $(y'_{(\alpha,\beta)})_{(\alpha,\beta)\in A\times B}$.
\end{remark}

\section{Net convergence structures}
\label{sec:ncs}

We are now in a position to introduce the main object of this
paper. Fix a set~$X$. By a \term{net convergence structure} on $X$ we
mean a function $\eta\colon X\to\mathcal P\bigl((\mathfrak N(X)\bigr)$
satisfying certain axioms that will be given below. Instead of
$(x_\alpha)_{\alpha\in\Lambda}\in\eta(x)$, we write
$(x_\alpha)_{\alpha\in\Lambda}\goeseta x$ and say that
$(x_\alpha)_{\alpha\in\Lambda}$ $\eta$-converges to~$x$. For
convenience, and when there is no risk of confusion, we may
de-emphasize $\eta$ and simply write $x_\alpha\to x$ and say
$(x_\alpha)$ converges to~$x$. In all these cases, $x$ is called a
\term{limit} of $(x_\alpha)_{\alpha\in\Lambda}$. Here are the axioms:
\begin{enumerate}
\item[(N1)] Constant nets converge: if $x_\alpha=x$ for every
  $\alpha$ then $x_\alpha\to x$;
\item[(N2)] If a net converges to $x$ then every
  quasi-subnet of it converges to~$x$;
\item[(N3)] Suppose that $(x_\alpha)_{\alpha\in\Lambda}\to x$
  and $(y_\alpha)_{\alpha\in\Lambda}\to x$. Let
  $(z_\alpha)_{\alpha\in\Lambda}$ be a net in $X$ such that
  $z_\alpha\in\{x_\alpha,y_\alpha\}$ for every~$\alpha$. Then
  $z_\alpha\to x$.
\end{enumerate}
In this case, we say that $X$ together with this net convergence
structure is a \term{net convergence space}. If we wish to highlight
both the set and the convergence on it, we will usually write
$(X,\xrightarrow{\eta})$ or $(X,\to)$.

If $x_\alpha\to x$ and $(x_\alpha)\sim(y_\beta)$ for some admissible
nets $(x_\alpha)$ and $(y_\beta)$, then $y_\beta\to x$; this follows
immediately from the fact that $(y_\beta)$ is a quasi-subnet
of~$(x_\alpha)$. In particular, an admissible net converges to $x$ iff
some (and therefore every) tail of it converges to~$x$. This also means
that one may think of convergence of equivalence classes in
$\mathfrak N(X)/\!\!\sim$ rather than just for individual nets.

We now extend convergence to nets that need not be admissible:
for an arbitrary net $(x_\alpha)$ in~$X$ we say $x_\alpha\to x$
if $(x_\alpha)$ is tail equivalent to an admissible net that converges
to~$x$. However, every time we quantify over ``all nets in $X$'', we
implicitly mean ``over $\mathfrak N(X)$''.

Given two convergence structures $\eta_1$ and $\eta_2$ on~$X$, we say
that $\eta_1$ is \term{stronger} than $\eta_2$ if
$x_\alpha\xrightarrow{\eta_1}x$ implies
$x_\alpha\xrightarrow{\eta_2}x$ for every (admissible) net
$(x_\alpha)$ in~$X$; in this case we also say that $\eta_2$ is
\term{weaker} than~$\eta_1$.

Every convergence induced by a topology satisfies the
axioms above; such convergences will be referred to as
\term{topological}.  We may now extend many properties of
convergence from topology to the setting of general net convergence spaces.

A net convergence space is said to be \term{Hausdorff} if every
net has at most one limit. If the structure is induced
by a topology, this just means the topology is Hausdorff.

It follows from Proposition~\ref{subnet-q-subnet} that
the axiom (N2) is equivalent to the following two alternative
conditions:
\begin{enumerate}
\item [(N2a)] If a net converges to $x$, then so does each of its subnets, and
\item [(N2b)] If $(x_\alpha)$ converges to $x$ and $(y_\beta)$ is tail
equivalent to $(x_\alpha)$ then $(y_\beta)$ converges to $x$.
\end{enumerate}

\begin{example}\label{N1N2notN3}
  The following example shows that (N1) and (N2) do not imply (N3),
  even in the Hausdorff case. For a net $(x_\alpha)$ in~$\mathbb R$,
  we define $x_\alpha\to x$ if $(x_\alpha)$ converges to $x$ in
  the usual convergence of $\mathbb R$ and, in addition, $(x_\alpha)$
  either has a tail consisting of only rational numbers or only of
  irrational numbers. It is easy to see that this convergence
  satisfies (N1) and (N2) and is Hausdorff, yet it fails (N3). 
\end{example}

\begin{example}
There are important convergences that fail to form net convergence
structures. For example, convergence along an ultrafilter
fails (N2).
\end{example}

Let $X$ be a net convergence space and $A\subseteq X$.  We denote the
set of limits of all nets in $A$ by~$\overline{A}$ and call it the
\term{closure} or the \term{adherence} of~$A$. If the limit of every
convergent net in $A$ belongs to~$A$, we say that $A$ is \term{closed}
and write $\overline{A}=A$. A word of warning: the closure of a set
need not be closed! We say that $A$ is \term{dense} if
$\overline{A}=X$. For $x\in A\subseteq X$ we say that $A$ is a
\term{neighborhood} of $x$ if for every net in $X$ that converges
to~$x$, a tail of the net is contained in~$A$.  $A$ is said to be
\term{open} if it is a neighborhood of each of its points. A function
$f\colon X\to Y$ between two net convergence spaces is said to be
\term{continuous at} $x$ if $x_\alpha\to x$ in $X$ implies
$f(x_\alpha)\to f(x)$ in $Y$ for every admissible net $(x_\alpha)$
in~$X$. Note that the net $\bigl(f(x_\alpha)\bigr)$ need not be
admissible in~$Y$; it suffices that it is tail equivalent to an
admissible net in $Y$ that converges to~$f(x)$.
We say that $f$ is
\term{continuous} if it is continuous at every $x\in X$.

Let $\mathcal A$ be a family of functions from a set $X$ to a
collection of net convergence spaces. We now consider the net
convergence structure on $X$ induced by $\mathcal A$: put
$x_\alpha\to x$ whenever $f(x_\alpha)\to f(x)$ for every
$f\in\mathcal A$. It can be easily verified that this is the weakest
net convergence structure that makes every function in $\mathcal A $
continuous.

Let $X$ be a net convergence space and $Y\subseteq X$. There is a
natural way to define a net convergence structure on~$Y$: for a net
$(x_\alpha)$ in $Y$ and a point $x$ in~$Y$ we put $x_\alpha\to x$ in
$Y$ whenever $x_\alpha\to x$ in the original convergence on~$X$. We
call $Y$ a \term{subspace} of~$X$. Equivalently, this is the
convergence structure on $Y$ induced by the inclusion map of $Y$
into~$X$; i.e., it is the weakest net convergence structure on $Y$
that makes the inclusion into $X$ continuous.

$X$ is said to be \term{compact} if every net in $X$ has a convergent
quasi-subnet. Then we say that a subset $A$ of $X$ is compact if it is
compact with respect to the induced convergence structure;
equivalently, if every net in $A$ has a quasi-subnet that converges in
$X$ to an element of~$A$. Several features of compact subsets from
topology continue to hold in the setting of net convergence spaces.
For example, every compact subspace of a Hausdorff space is closed. It
is also easy to see that every real- or complex-valued continuous
function is bounded on every compact subset of~$X$. A net convergence
space is \term{locally compact} if every convergent net has a tail
contained in a compact set.

Let $(X_\gamma)_{\gamma\in\Gamma}$ be a family of net convergence
spaces and consider the Cartesian product
$X=\prod_{\gamma\in\Gamma}X_\gamma$. For every $\gamma\in\Gamma$ let
$P_\gamma\colon X\to X_\gamma$ be the canonical coordinate map. The
\term{product net convergence structure} on $X$ is given by
$x_\alpha\to x$ in $X$ whenever $P_\gamma x_\alpha\to P_\gamma x$ in
$X_\gamma$ for every~$\gamma$. Equivalently, this is the net
convergence structure on $X$ induced by the family of all coordinate
maps $\{P_\gamma\}_{\gamma\in\Gamma}$; we view this as the
``point-wise'' or ``coordinate-wise'' convergence.

All the concepts introduced in the preceding paragraphs agree with the
appropriate topological concepts when the convergence is
topological. We list several of these properties in the following
proposition, all of which have well-known analogues for filter
convergence structures. In Section~\ref{sec:equiv} we will prove that
our concept of a net convergence structure is equivalent to the
concept of a filter convergence structure; this will make
Proposition~\ref{bas} redundant. Nevertheless, we provide a proof of
the proposition in order to illustrate the simplicity of net
convergence techniques.

\begin{proposition}\label{bas}
  The following are valid in every net convergence structure:
  \begin{enumerate}
  \item\label{bas-cont-comp} The composition of two continuous
    functions is continuous;
  \item\label{bas-cl-subset} Every set is contained in its closure;
  \item\label{bas-compl} A set is open iff its complement is closed;
  \item\label{bas-cl-inters} The intersection of any collection of
    closed sets is closed; 
  \item\label{bas-cl-union} The union of any finite collection of
    closed sets is closed;
  \item\label{bas-op-union} The union of any collection of open sets
    is open;
  \item\label{bas-op-inters} The intersection of any finite collection of
    open sets is open.
  \end{enumerate}
\end{proposition}

\begin{proof}
  \eqref{bas-cont-comp}, \eqref{bas-cl-subset}, \eqref{bas-cl-inters}
   are straightforward.

  To prove~\eqref{bas-compl}, first suppose that $A$ is open. Let
  $(x_\alpha)$ be a net in $A^C$ such that $x_\alpha\to x$, yet $x\in
  A$. Then a tail of the net is in~$A$, which is a contradiction.

  Conversely, if $A^C$ is closed, $(x_\alpha)_{\alpha\in\Lambda}$
  converges to $x\in A$, and no tail of the net is contained in~$A$,
  then let $\Lambda_0=\{\alpha\in\Lambda\mid x_\alpha\in A^C\}$. It is
  easy to see that $\Lambda$ is directed and the restriction of
  $(x_\alpha)_{\alpha\in\Lambda}$ to $\Lambda_0$ is a quasi-subnet of
  $(x_\alpha)_{\alpha\in\Lambda}$. This quasi-subnet is contained in
  $A^C$ and still converges to $x$ by (N2). Since this implies
  $x\in A^C$, we have reached a contradiction.

  \eqref{bas-op-union} follows from \eqref{bas-compl} and
  \eqref{bas-cl-inters}.

  \eqref{bas-op-inters} It suffices to prove the statements for two
  sets and apply induction. Let $A_1$ and $A_2$ be two open sets; we
  need to show that $A_1\cap A_2$ is open. Suppose that
  $x\in A_1\cap A_2$ and $x_\alpha\to x$. Since $x\in A_1$ and $A_1$
  is open, there exists $\alpha_1$ such that the tail set
  $\{x_\alpha\}_{\alpha\ge\alpha_1}$ is contained in~$A_1$. Similarly,
  some tail set $\{x_\alpha\}_{\alpha\ge\alpha_2}$ is contained
  in~$A_2$. Now choose any $\alpha_0\ge\alpha_1,\alpha_2$ and note
  that the tail set $\{x_\alpha\}_{\alpha\ge\alpha_0}$ is contained in
  $A_1\cap A_2$.

  \eqref{bas-cl-union} follows from \eqref{bas-compl} and
  \eqref{bas-op-inters}.
\end{proof}

Even though the terminology resembles that from topology, note that
there are important differences. In particular, in a convergence
structure, the closure of a set need not be closed:

\begin{example}
  The following example is derived from \cite{Beattie:02}. Let
  $X=\mathbb R$ and define a convergence $x_\alpha\to x$ as follows:
  if $x\ne 0$ put $x_\alpha\to x$ whenever $(x_\alpha)$ converges to
  $x$ in the usual topology of~$\mathbb R$; otherwise, define
  $x_\alpha\to 0$ whenever $(x_\alpha)$ converges to $0$ in the usual
  topology of $\mathbb R$ and a tail of the net is contained
  in~$\mathbb Q$. Observe that
  $\overline{\mathbb R\setminus\mathbb Q}=\mathbb
  R\setminus\{0\}$. The latter set is not closed since its closure
  is~$\mathbb R$. In particular, this shows that changing the
  convergence at a single point is enough to spoil a topological
  convergence.
\end{example}

There are several natural ways in which a given net convergence
structure induces another convergence. We call the new structure a
\term{modification} of the original one. Let $X$ be a net convergence
space. It follows from Proposition~\ref{bas} that the collection of
all open subsets of $X$ forms a topology, say,~$\tau$. We call $\tau$
the topology associated with the convergence structure. This topology
induces its own convergence, which we will denote
$x_\alpha\goestau x$; it is called the \term{topological modification}
of the original convergence structure.

\begin{proposition}\label{tmod-finest}
  Let $(X,\to)$ be a convergence space and $\tau$ its topological
  modification. Then $x_\alpha\to x$ implies $x_\alpha\goestau
  x$. Moreover, $\tau$ is the finest topology on $X$ whose convergence
  is weaker than the original convergence.
\end{proposition}

\begin{proof}
  If $x_\alpha\to x$ then, by the definition of an open set, every open
  neighborhood of $x$ contains a tail of $(x_\alpha)$ and hence
  $x_\alpha\goestau x$.

  Let $\sigma$ be a topology on $X$ such that $x_\alpha\to x$
  implies $x_\alpha\xrightarrow{\sigma}x$. Let $A$ be a
  $\sigma$-closed set. If $(x_\alpha)$ is a net in $A$ and
  $x_\alpha\to x$, then $x_\alpha\xrightarrow{\sigma}x$; so
  $x\in A$. It follows that $A$ is closed in $(X,\to)$, so it is
  $\tau$-closed.
\end{proof}

It follows that the topological modification is generally weaker than the
original convergence. We will see in Example~\ref{otop-nonHaus} that
it may be dramatically weaker.  Clearly, a net convergence structure
is topological precisely when it equals its topological modification.

\subsection*{Net convergence vector spaces}
Throughout the rest of the section, let $X$ be a vector space
over~$\mathbb K$, where $\mathbb K$ is either $\mathbb R$
or~$\mathbb C$, equipped with the usual convergence. A
net convergence structure on $X$ is said to be \term{linear} if
addition and scalar multiplication operations are jointly
continuous. We then say that $X$ is a \term{net convergence vector
  space}. It is easy to see that $x_\alpha\to x$ iff
$x_\alpha-x\to 0$.

A linear operator $T\colon X\to Y$ between two net convergence vector
spaces is said to be an \term{isomorphism} if $T$ is a bijection and
both $T$ and $T^{-1}$ are continuous. We say that $T$ is an
\term{isomorphic embedding} if $T$ is an isomorphism when viewed as a
map from $X$ to $\Range T$.

\begin{definition}\label{def:bdd}
  A subset $A$ of net convergence vector space $X$ is said to be
  \term{bounded} if $(r_\alpha x)_{(\alpha,x)}\to 0$ whenever
  $(r_\alpha)_{\alpha\in\Lambda}\to 0$ in~$\mathbb K$. In this definition
  $(r_\alpha x)_{(\alpha,x)}$ is interpreted as a net over the index
  set $\Lambda\times A$ directed via $(\alpha,x)\le(\beta,y)$ whenever
  $\alpha\le\beta$. With a minor abuse of notation, we will denote
  this net by $(r_\alpha A)$.  A net convergence vector space is said
  to be \term{locally bounded} if every convergent net has a bounded
  tail.
\end{definition}

\begin{proposition}\label{bdd}
  Let $A$ and $B$ be subsets of a convergence vector space~$X$.
  \begin{enumerate}
  \item\label{bdd-sub} If $A$ is bounded and $B\subseteq A$ then $B$
    is bounded;
  \item\label{bdd-scal} If $A$ is bounded then $rA$ is bounded for
    every $r\ge 0$;
  \item\label{bdd-union} If $A$ and $B$ are both bounded then so are
    $A\cup B$ and $A+B$;
  \item\label{bdd-fin} Every finite set is bounded.
  \end{enumerate}
\end{proposition}

\begin{proof}
  The proofs of \eqref{bdd-sub} and~\eqref{bdd-scal} are
  straightforward. To prove~\eqref{bdd-union} we let $A$ and $B$ be two
  bounded sets in~$X$, $C=A\cup B$, and
  $(r_\alpha)_{\alpha\in\Lambda}\to 0$ in~$\mathbb K$. We may assume
  without loss of generality that both $A$ and $B$ are non-empty. As
  before, we view $\Lambda\times A$, $\Lambda\times B$, and
  $\Lambda\times C$ as directed sets under $(\alpha,x)\le(\beta,y)$
  whenever $\alpha\le\beta$. Fix any $a_0\in A$ and $b_0\in B$. Define
  two nets, one in $A$ and the other in~$B$, both indexed by
  $\Lambda\times C$ as follows: for $(\alpha,c)$ in $\Lambda\times C$,
  put
  \begin{displaymath}
    x_{(\alpha,c)}=
    \begin{cases}
      r_\alpha c &\mbox{ if }c\in A\\
      r_\alpha a_0 &\mbox{ otherwise}
    \end{cases}
    \qquad
     y_{(\alpha,c)}=
    \begin{cases}
      r_\alpha c &\mbox{ if }c\in B\\
      r_\alpha b_0 &\mbox{ otherwise.}
    \end{cases}
  \end{displaymath}Observe that
  \begin{math}
      (x_{(\alpha,c)})_{(\alpha,c)\in\Lambda\times C}
  \end{math}
  is
  strongly tail equivalent to the net $(r_\alpha A)$.
  Since $A$ is bounded, the latter net converges to zero and hence,
  $x_{(\alpha,c)}\to 0$. Similarly, we have $y_{(\alpha,c)}\to 0$. Observe
  that for every $(\alpha,c)$ in $\Lambda\times C$ we have
  \begin{displaymath}
    r_\alpha c=
    \begin{cases}
      x_{(\alpha,c)} & \mbox{ if }c\in A\\
      y_{(\alpha,c)} & \mbox{ if }c\in B.
    \end{cases}
  \end{displaymath}
  Axiom (N3) yields $(r_\alpha c)_{(\alpha,c)\in\Lambda\times C}\to
  0$, which shows $C$ is bounded.

  Let $D=A+B$. For every $d\in D$, choose $a\in A$ and $b\in B$ such
  that $d=a+b$, and put $a=f(d)$ and $b=g(d)$. Define two nets indexed by
  $\Lambda\times D$ as follows:
  \begin{math}
    x_{(\alpha,d)}=r_\alpha f(d)
  \end{math}
  and
  \begin{math}
    y_{(\alpha,d)}=r_\alpha g(d).
  \end{math}
  It is easy to see
  \begin{math}
    (x_{\alpha,d})_{\Lambda\times D}\preceq(r_\alpha A)
  \end{math}
  and
  \begin{math}
    (y_{\alpha,d})_{\Lambda\times D}\preceq(r_\alpha B)
  \end{math}
  hence
  $x_{\alpha,d}\to 0$ and $y_{\alpha,d}\to 0$. It follows that
  \begin{math}
    r_\alpha d=x_{\alpha,d}+y_{\alpha,d}\to 0
  \end{math}
  and, as a consequence, $D$ is bounded.
  
  \eqref{bdd-fin} It follows from the continuity of scalar
  multiplication that every singleton is bounded; now
  apply~\eqref{bdd-union}.
\end{proof}

Recall that the \term{circled hull} of a set $A$ is the set
$\{rx\mid \abs{r}\le 1,\ x\in A\}$. We say that $A$ is \term{circled}
if it equals its circled hull or, equivalently, $x\in A$ implies
$rx\in A$ whenever $\abs{r}\le 1$. Note that
if $A$ is bounded then so is its circled hull. While this fact can be
verified directly, it will be almost trivial in the filter language:
see page~\pageref{circ-bdd}. It is also easy to see that if $A$ is
circled then the net $(r_\alpha)$ in Definition~\ref{def:bdd} may be
taken to be the sequence $(\frac1n)_{n\in\mathbb N}$. That is, a
circled set $A$ is bounded iff $\frac1n A\to 0$. We also make use of
the fact that a bounded set $A$ has a bounded convex
hull: the latter is contained in $C+C$ where $C$ is the circled
hull of~$A$.

A net $(x_\alpha)_{\alpha\in\Lambda}$ in a net convergence vector
space is said to be \term{Cauchy} if the net
$(x_\alpha-x_\beta)_{(\alpha,\beta)\in\Lambda\times\Lambda}$ converges
to zero. Here $\Lambda\times\Lambda$ is directed coordinate-wise.  It
is easy to see that every convergent net is Cauchy. Indeed,
put $y_{(\alpha,\beta)}=x_\alpha$ and
$z_{(\alpha,\beta)}=x_\beta$. The nets $(y_{(\alpha,\beta)})$ and
$(z_{(\alpha,\beta)})$ are both strongly tail equivalent to
$(x_\alpha)$, so they both converge to~$x$. It follows that
  \begin{math}
    x_\alpha-x_\beta=y_{(\alpha,\beta)}-z_{(\alpha,\beta)}\to 0.
  \end{math}
A net convergence vector space is said to be \term{complete} if every
Cauchy net is convergent.

Let $X$ be a topological vector space. In this case, our definitions
of boundedness, Cauchy and completeness agree with the standard ones
given in \cite[pp.~44 and~56]{Kelley:76}. In fact, for the terms
Cauchy and completeness, our definitions are exactly the same. As for
boundedness, according to \cite[p.~44]{Kelley:76}, a subset $A$ of a
topological vector space is bounded if for every neighborhood of zero
$U$ there exists a $t>0$ such that $A\subseteq tU$. It is easy to see
that this agrees with our Definition~\ref{def:bdd}.
%

Let $(X,\to)$ be a Hausdorff net convergence vector space. We define a
new net convergence structure on $X$ as follows: for a
net $(x_\alpha)$, we write $x_\alpha\goesm x$ if there exists a
bounded set $B$ such that for every $\varepsilon>0$ there is an
$\alpha_0$ such that $x_\alpha-x\in\varepsilon B$ whenever
$\alpha\ge\alpha_0$. It can be easily verified that $(X,\goesm)$ is
indeed a linear net convergence structure on~$X$; it is called the
\term{Mackey modification} of the original structure.

\begin{proposition}
  If $x_\alpha\goesm x$ then $x_\alpha\to x$.
\end{proposition}

\begin{proof}
  Since both structures are linear, we may assume that
  $x=0$. Let $B$ be a bounded set that witnesses
  $x_\alpha\goesm 0$ as above. The net $(\frac1n B)$ 
  as in Definition~\ref{def:bdd} satisfies $\frac1nb\to 0$. Now
  observe that $(x_\alpha)$ is a quasi-subnet of $(\frac1n B)$.
\end{proof}

\begin{proposition}\label{Mack-bdd}
  Let $T\colon X\to Y$ be a linear map between net convergence vector
  spaces $X$ and~$Y$. The following are equivalent:
  \begin{enumerate}
  \item\label{Mack-bdd-bdd} $T$ is bounded; that is, $T$ maps bounded
    sets to bounded sets; 
  \item\label{Mack-bdd-mm} $T$ is Mackey-continuous:
    $x_\alpha\goesm x$ implies $Tx_\alpha\goesm Tx$;
  \item\label{Mack-bdd-m} $x_\alpha\goesm x$ implies $Tx_\alpha\to Tx$.
  \end{enumerate}
\end{proposition}

\begin{proof}
  \eqref{Mack-bdd-bdd}$\Rightarrow$\eqref{Mack-bdd-mm} is straightforward; 
  \eqref{Mack-bdd-mm}$\Rightarrow$\eqref{Mack-bdd-m} follows from the
  preceding proposition.

  \eqref{Mack-bdd-m}$\Rightarrow$\eqref{Mack-bdd-bdd} Suppose that $B$
  is a bounded subset of~$X$. It remains to show that $T(B)$ is
  bounded. We may assume that~$B$, and therefore $T(B)$, is circled. Let
  $(r_\alpha)_{\alpha\in\Lambda}\to 0$ in~$\mathbb R$. Since $B$ is
  bounded, we have $(r_\alpha b)\goesm 0$ where $(r_\alpha b)$ is indexed by
  $\Lambda\times B$ directed by the first component. It follows that
  $(r_\alpha Tb)\to 0$, so $T(B)$ is bounded.
\end{proof}

A linear convergence structure is \term{Mackey} if it agrees with its
Mackey modification.

\section{Filter convergence structures}
\label{sec:fcs}

This section contains a brief review of terminology from filter
convergence structures; we refer the reader to~\cite{Beattie:02} for a
detailed exposition.

We begin by fixing some notation. For a set~$X$, a non-empty
collection $\mathcal B$ of non-empty subsets of $X$ is said to be a
\term{filter base} if it is directed with respect to reverse
inclusion: for all $A,B\in\mathcal B$ there exists $C\in\mathcal B$
such that $C\subseteq A\cap B$. A non-empty collection $\mathcal F$ of
non-empty subsets of $X$ is said to be a \term{filter} if it is closed
under taking supersets and finite intersections. It is easy to see
that every filter is a filter base. For a filter base $\mathcal B$, we
use~$[\mathcal B]$ to denote the collection of all supersets of sets
in~$\mathcal B$; this forms a filter and is called the \term{filter
  generated by}~$\mathcal B$. We write $\mathfrak F(X)$ for the set of
all filters on~$X$.

A filter $\mathcal U$ is said to be an \term{ultrafilter} if it is not
properly contained in another filter. Equivalently, for every
$A\subseteq X$, either $A\in\mathcal U$ or $A^C\in\mathcal U$. It
follows from Zorn's Lemma that every filter is contained in an
ultrafilter. Any singleton $\{x\}$ is clearly a filer base; in this
case $[\{x\}]$ is denoted by~$[x]$. It is easy to verify this is an
ultrafilter.

Let $X$ be a set and $\lambda$ be a map from $X$ to
$\mathcal P\bigl(\mathfrak F(X)\bigr)$. We write
$\mathcal F\xrightarrow{\lambda} x$ instead of
$\mathcal F\in\lambda(x)$, or $\mathcal F\to x$ when there is no risk
of confusion, and say that $\mathcal F$ converges to~$x$. In any case
$x$ is called a \term{limit} of~$\mathcal F$. We say that $\lambda$ is
a \term{filter convergence structure} if the following axioms are
satisfied:
\begin{itemize}
  \item [(F1)] $[x]\to x$ for every $x\in X$;
  \item [(F2)] If $\mathcal F\to x$ and $\mathcal F\subseteq\mathcal
    G$ then $\mathcal G\to x$;
  \item [(F3)] If $\mathcal F\to x$ and $\mathcal G\to x$ then
    $\mathcal F\cap\mathcal G\to x$.
\end{itemize}
A set equipped with a filter convergence structure is called a
\term{filter convergence space}.

Every topology on $X$ gives rise to a natural filter convergence structure. It
is easy to see that for every $x\in X$ the collection of all
neighborhoods of $x$ forms a filter; we will denote it
by~$\mathcal N_x$. Consider a filter $\mathcal F$ convergent to
$x$ if $\mathcal N_x \subseteq \mathcal F$. It is easy to see that
this defines a filter convergence structure on~$X$; any filter
convergence structure that arises in this way is called \term{topological}.

Let $X$ be a filter convergence space. We say that $X$ is
\term{Hausdorff} if every filter has at most one limit. For $x\in X$,
the filter $\mathcal U_x=\bigcap\{\mathcal F\mid\mathcal F\to x\}$ is
called the \term{neighborhood filter} of~$x$, and any set $U\in\mathcal U_x$
is called a \term{neighborhood} of~$x$. A subset $A$ of $X$ is
\term{open} if it is a neighborhood of each of its points. Again,
given $A\subseteq X$, we define
\begin{displaymath}
  \overline{A}=\bigl\{x\in X\mid \mathcal F\to x
  \mbox{ for some filter }\mathcal F\mbox{ such that }
  A\in\mathcal F\bigr\}.
\end{displaymath}
We call $\overline{A}$ the \term{closure} or the \term{adherence}
of~$A$. Naturally, we say that $A$ is \term{closed} if
$\overline{A}=A$ and \term{dense} if $\overline{A}=X$. It can be
verified that $A$ is open iff $A^C$ is closed; see Lemma~1.3.4
in~\cite{Beattie:02}. It can also be easily verified that the
collection of all open sets in $X$ form a topology. As outlined above,
this topology induces a natural filter convergence structure on~$X$;
we call it the \term{topological modification} of the original filter
convergence structure.

Let $f\colon X\to Y$ be a function between two filter convergence
spaces. For every filter $\mathcal F$ on~$X$, the collection
$\{f(A)\mid A\in\mathcal F\}$ is a filter base in~$Y$; the filter
that it generates is called the
\term{image filter} of $\mathcal F$ under $f$ and is denoted by $f(\mathcal
F)$. We say that $f$ is \term{continuous} at $x$ if $\mathcal F\to x$
implies $f(\mathcal F)\to f(x)$. 

Let $\mathcal A$ be a family of functions from a set $X$ to filter
convergence spaces. For every $x\in X$, put $\mathcal F\to x$ in $X$
whenever $f(\mathcal F)\to f(x)$ for every $f\in\mathcal A$. This is
the weakest filter convergence structure that makes all
$f\in\mathcal A$ continuous. This allows us to define the
\term{subspace filter convergence structure} on a subset of a filter
convergence space as the weakest structure that makes the inclusion
map continuous. Likewise, we define the \term{product filter
  convergence structure} on the Cartesian product of filter convergent
spaces as the weakest structure that makes all the coordinate
projections continuous.

A filter convergence space $X$ is said to be \term{compact} if every
ultrafilter converges. A subset $A$ of $C$ is said to be compact if it
is compact in the subspace filter convergence structure induced
by~$X$. We say that $X$ is \term{locally compact} if every convergent
filter on $X$ contains a compact set.


A filter convergence space is \term{first countable at} $x$ if
$\mathcal F\to x$ implies that there exists a filter
$\mathcal G$ with a countable base such that
$\mathcal G\subseteq\mathcal F$ and $\mathcal G\to x$. A filter
convergence space is said to be \term{first countable} if it is first
countable at every point.

We now take $X$ to be a vector space over~$\mathbb K$. A filter
convergence structure on $X$ is said to be linear if the operations of
addition and scalar multiplication are jointly continuous; here
$\mathbb K$ is equipped with the filter convergence structure induced
by the usual topology of~$\mathbb K$. We then call $X$ a \term{filter
  convergence vector space}.

Let $X$ be a filter convergence vector space and $A\subseteq X$. Let
$\mathcal N_0$ denote the filter of all neighborhoods of zero
in~$\mathbb K$ and put
\begin{math}
  \mathcal N_0A=\{UA\mid U\in\mathcal N_0\}.
\end{math}
It is easy to see that $\mathcal N_0A$ is a filter base. $A$ is said
to be a \term{bounded} set if $[\mathcal N_0A]\to 0$ in~$X$. The space
$X$ is \term{locally bounded} if every convergent filter contains a
bounded set.
\label{circ-bdd}
Let $A$ be a bounded set and use $C$ to denote
its circled hull. It is easy to see that
$\mathcal N_0A=\mathcal N_0C$, so $C$ is also bounded.

%
%

Let $X$ be a filter convergence vector space and $\mathcal F$ a
filter on~$X$. The collection
\begin{math}
  \{A-B\mid A,B\in\mathcal F\}
\end{math}
is a filter base; we write $\mathcal F-\mathcal F$ for the filter
generated by this collection. $\mathcal F$ is called
\term{Cauchy} if $\mathcal F-\mathcal F\to 0$ and we say that $X$ is
\term{complete} if every Cauchy filter on $X$ converges.

For a filter $\mathcal F$ on a filter convergence vector space~$X$,
define $\mathcal F\goesm 0$ if $\mathcal F\supseteq\mathcal N_0B$ for
some bounded set~$B$. It follows that $\mathcal F\goesm x$ if
$\mathcal F-x\goesm 0$ defines a new linear filter convergence
structure on~$X$, called the \term{Mackey modification} of the
original structure.

\section{Equivalence of net and filter structures}
\label{sec:equiv}

We now show that the concepts of net and filter convergence structures
are equivalent. Throughout this section $X$ denotes an arbitrary set,
unless specified otherwise.

Observe that for a net $(x_\alpha)$, its collection of tail sets
$\langle x_\alpha\rangle$ is a filter base; we will call it the
\term{tail filter base} of $(x_\alpha)$. The filter generated by this
filter base will be called the \term{tail filter} of $(x_\alpha)$ and
denoted by $[x_\alpha]$. The definition of strong tail equivalence
between two nets means they have the same tail filter
bases. This is a stronger property than
two nets being tail equivalent, which can be characterized in terms of
tail filters: $(x_\alpha)\preceq(y_\beta)$ iff
$[y_\beta]\subseteq[x_\alpha]$. Thus, two nets are tail
equivalent iff their tail filters agree.

\begin{proposition}\label{tfb-net}
  Every filter base on $X$ is the tail filter base of some admissible
  net in~$X$.
\end{proposition}

\begin{proof}
  Let $\mathcal B$ be a filter base on~$X$. Let $\Lambda$ be the
  subset of $\mathcal B\times X$ defined by $(A,x)\in\Lambda$ whenever
  $x\in A$:
  \begin{displaymath}
    \Lambda=\bigl\{(A,x)\mid x\in A\in\mathcal B\bigr\}.
  \end{displaymath}
  A pre-order can be defined on $\Lambda$ via $(A,x)\le(B,y)$ when $B\subseteq
  A$; this makes $\Lambda$ a directed set. Define a net
  indexed by $\Lambda$ via $x_{(A,x)}=x$. It is easy to see that this
  net is admissible and its tail filter base equals~$\mathcal B$.
\end{proof}

We are now ready to prove Theorem~\ref{admis}:

\begin{proof}[Proof of Theorem~\ref{admis}.]
  Let $(x_\alpha)$ be a net in~$X$. By Proposition~\ref{tfb-net},
  there is an admissible net whose tail filter base is
  $\langle x_\alpha\rangle$, hence this new net is strongly tail
  equivalent to $(x_\alpha)$.
\end{proof}

Proposition~\ref{tfb-net} implies that there is a natural bijection
between $\mathfrak N(X)/\!\!\sim$ and $\mathfrak F(X)$. Indeed, it
follows from Proposition~\ref{tfb-net} that the map induced by
associating every net in $\mathfrak N(X)$ to its tail filter is a map
from $\mathfrak N(X)$ onto $\mathfrak F(X)$. By the definition of tail
equivalence, this map induces a bijection from
$\mathfrak N(X)/\!\!\sim$ to $\mathfrak F(X)$.

The next two theorems show that the bijection between
$\mathfrak N(X)/\!\!\sim$ and $\mathfrak F(X)$ induces a bijection
between the set of all net convergence structures and the set of all
filter convergence structures on~$X$. This will help us show that
Section~\ref{sec:ncs} and Section~\ref{sec:fcs} are exact translations
of each other.

\begin{theorem}\label{ncs-fcs}
  Let $X$ be a net convergence space. Then $X$ is also a filter convergence
  space under the filter convergence structure given by $\mathcal F\to
  x$ whenever $\mathcal F$ is the tail filter of some net converging
  to~$x$.
\end{theorem}

\begin{proof}
  We need to verify (F1)--(F3). (F1) is trivial because, for every
  $x\in X$, $[x]$ is the tail filter of any
  constant net in $X$ with constant value~$x$, and such a net converges to
  $x$ by (N1).

  To show (F2), suppose $\mathcal F\to x$ and $\mathcal
  F\subseteq\mathcal G$. It follows from $\mathcal F\to~x$ that
  $\mathcal F=[x_\alpha]$ for some net $(x_\alpha)$ where
  $x_\alpha\to x$. Proposition~\ref{tfb-net} allows us to find
  a net $(y_\beta)$ in $\mathfrak N(X)$ such that
  $\mathcal G=[y_\beta]$. Now $\mathcal
  F\subseteq\mathcal G$ implies $(y_\beta)\preceq(x_\alpha)$. Therefore,
  $y_\beta\to x$ by (N2), and hence $\mathcal G\to x$.

  Finally, we prove (F3). Assume $\mathcal F\to x$ and $\mathcal G\to
  x$. Then we can find two nets $(x_\alpha)_{\alpha\in\Lambda}$ and
  $(y_\alpha)_{\alpha\in\Lambda}$ in $\mathfrak N(X)$ satisfying
  $\mathcal F=[x_\alpha]$, $\mathcal G=[y_\alpha]$, $x_\alpha\to x$,
  and $y_\alpha\to x$; note that we may assume the two nets
  have the same index set by Remark~\ref{same-ind}. Letting
  $\Gamma=\Lambda\times\{1,2\}$ and defining $(\alpha,i)\le(\beta,j)$
  in $\Gamma$ if $\alpha\le\beta$ in $\Lambda$ makes $\Gamma$
  into a directed set. Define a net $(z_\gamma)_{\gamma\in\Gamma}$ by 
  setting $z_{(\alpha,1)}=x_\alpha$ and $z_{(\alpha,2)}=y_\alpha$. It
  is easy to see that both $(x_\alpha)$ and $(y_\alpha)$ are
  quasi-subnets of $(z_\gamma)$ and, as a consequence,
  $[z_\gamma]\subseteq[x_\alpha]\cap[y_\alpha]=\mathcal F\cap\mathcal
  G$. It now suffices to verify $z_\gamma\to x$ because this would
  imply $[z_\gamma]\to x$ and, therefore,
  $\mathcal F\cap\mathcal G\to x$ by (F2).

  Define two auxiliary nets $(\bar x_\gamma)_{\gamma\in\Gamma}$ and
  $(\bar y_\gamma)_{\gamma\in\Gamma}$ as follows: put
  $\bar x_{(\alpha,i)}=x_\alpha$ and $\bar
  y_{(\alpha,i)}=y_\alpha$. It is easy to see $(\bar x_\gamma)$ and
  $(\bar y_\gamma)$ are strongly tail equivalent to $(x_\alpha)$ and
  $(y_\alpha)$, respectively, so $\bar x_\gamma\to x$ and $\bar
  y_\gamma\to x$. Since $z_{(\alpha,1)}=\bar x_{(\alpha,1)}$ and
  $z_{(\alpha,2)}=\bar y_{(\alpha,2)}$ for every~$\alpha$, (N3) yields
  $z_\gamma\to z$.
\end{proof}

\begin{theorem}\label{fcs-ncs}
  Let $X$ be a filter convergence space. Then it is also a net
  convergence space under the net convergence structure given by
  $x_\alpha\to x$ whenever $[x_\alpha]\to x$.
\end{theorem}

\begin{proof}
  It is easy to see that axiom (F1) of the original filter convergence
  implies (N1) for the new net convergence. To verify (N2), suppose
  $x_\alpha\to x$ and $(y_\beta)$ is a quasi-subnet of
  $(x_\alpha)$. Then $[x_\alpha]\to x$ and
  $[x_\alpha]\subseteq[y_\beta]$ implies $[y_\beta]\to x$ by (F2),
  which yields $y_\beta\to x$.

  To show (N3), we assume $x_\alpha\to x$, $y_\alpha\to x$, and let
  $(z_\alpha)$ be such that $z_\alpha\in\{x_\alpha,y_\alpha\}$ for
  every~$\alpha$. Let $A\in[x_\alpha]\cap[y_\alpha]$. It follows that
  $A$ contains a tail of $(x_\alpha)$ and a tail of $(y_\alpha)$. Then
  there exists $\alpha_0$ such that both
  $\{x_\alpha\}_{\alpha\ge\alpha_0}$ and
  $\{y_\alpha\}_{\alpha\ge\alpha_0}$ are subsets of~$A$. This shows
  $\{z_\alpha\}_{\alpha\ge\alpha_0}$ is a subset of~$A$, so
  $A\in[z_\alpha]$. It follows that
    \begin{math}
      [x_\alpha]\cap[y_\alpha]\subseteq[z_\alpha].
    \end{math}
    Since $[x_\alpha]\to x$ and $[y_\alpha]\to x$,
    (F3) implies $[x_\alpha]\cap[y_\alpha]\to x$. Now (F2) yields
    $[z_\alpha]\to x$, which implies $z_\alpha\to x$.
\end{proof}

Theorem~\ref{ncs-fcs} yields a map from net convergence structures on
$X$ to filter convergence structures on~$X$, while
Theorem~\ref{fcs-ncs} yields a map in the opposite direction. Thus,
every net convergence structure gives rise to an associated filter
convergence structure, and every filter convergence structure gives
rise to an associated net convergence structure. It is easy to verify
that these two maps are inverses of each other. Going forward, we will
use the term \term{convergence structure} instead of {\it net}
convergence structure or {\it filter} convergence
structure. Similarly, we will now talk about \term{convergence spaces}
and \term{convergence vector spaces}. We will see that the net and
filter languages each have their own advantages and disadvantages, so
one should use the language that is most convenient in a given
situation.

In~\cite{Ho:10}, the authors consider a definition of net
convergence structure consisting of just two axioms that are
analogous to our axioms (N1) and (N2). However, their definition does
not include (N3). They observe that their definition is implied by
filter convergence structures and ask whether the two approaches are
equivalent. We now know that they are not: (N3) is needed to ensure
the equivalence: Example~\ref{N1N2notN3} shows that (N1) and (N2) are
not sufficient.

\section{Equivalence of net and filter terminologies}
\label{sec:terms}

In Sections~\ref{sec:ncs} and~\ref{sec:fcs} we introduced various
terms in the settings of net and filter convergence structures.  We
will now show that translating between the net and filter languages
preserves the meaning of all these terms. For example, a set in a net
convergence vector space $X$ is bounded iff it is bounded in the
associated filter convergence structure on~$X$.

Throughout this section we take $X$ to be a convergence space. In
order to distinguish between properties, we will use prefixes ``net''
and ``filter''. For example, we write ``net-bounded'' versus
``filter-bounded'' to distinguish boundedness in the two settings. As
before, ``net'' will usually mean ``admissible net'' when quantifying
over all nets.

Most of these results are easy exercises, but we present their proofs
for the convenience of the reader.

\smallskip

\emph{Hausdorff:} Suppose $X$ is filter-Hausdorff, and assume
$x_\alpha\to x$ and $x_\alpha\to y$. Then $[x_\alpha]\to x$ and
$[x_\alpha]\to y$ implies $x=y$; hence, $X$ is
net-Hausdorff. Conversely, if $X$ is net-Hausdorff and
$\mathcal F\to x$ and $\mathcal F\to y$, find a net $(x_\alpha)$ such
that $\mathcal F=[x_\alpha]$. It follows that $x_\alpha\to x$ and
$x_\alpha\to y$, so $x=y$, confirming that $X$ is filter-Hausdorff.

\smallskip

\emph{Continuous functions:} Let $f\colon X\to Y$ be a function
between two convergence spaces. It is easy to see
$f\bigl([x_\alpha]\bigr)=\bigl[f(x_\alpha)\bigr]$ for every net
$(x_\alpha)$ in~$X$. If $x_\alpha\to x$ in $X$ and $f$ is filter-continuous, then $[x_\alpha]\to x$ and
\begin{math}
  \bigl[f(x_\alpha)\bigr]=f\bigl([x_\alpha]\bigr)\to f(x).
\end{math}
These imply $f(x_\alpha)\to f(x)$ and, consequently, that $f$ is net
continuous. Suppose now that $f$ is net continuous and
$\mathcal F\to x$ in~$X$.
Find a net $(x_\alpha)$ in $X$
such that $x_\alpha\to x$ and $\mathcal F=[x_\alpha]$. It follows that
$f(x_\alpha)\to f(x)$, so
\begin{math}
  f(\mathcal F)=\bigl[f(x_\alpha)\bigr]\to f(x),
\end{math}
and $f$ is filter-continuous.

All the properties that were defined using continuous functions are
then the same in net and filter structures. These include
\emph{subspaces}, \emph{product spaces}, and \emph{convergence vector spaces}.

\smallskip

\emph{Closure:} Let $x$ belong to the net-closure of~$A$. Then there
is a net $(x_\alpha)$ in $A$ such that $x_\alpha\to x$. Then
$[x_\alpha]\to x$. It follows from $A\in[x_\alpha]$ that $x$ is in the
filter-closure of~$A$. Conversely, suppose that $x$ is in the
filter-closure of~$A$. Then $A\in\mathcal F\to x$ for some
filter~$\mathcal F$. It follows that there is a net $(x_\alpha)$ such
that $\mathcal F=[x_\alpha]$ and $x_\alpha\to x$. Then
$A\in[x_\alpha]$ implies $A$ contains a tail of $(x_\alpha)$. Since this
tail still converges to~$x$, we conclude that $x$ is in the
net-closure of~$A$.

It now follows that the concepts of \emph{closed}, \emph{open}, and
\emph{dense} sets, as well as \emph{topological modification}, agree
in net and filter structures.

\smallskip

\emph{Neighborhoods:} we leave it as an exercise that
net-neighborhoods agree with filter-neighborhoods.

\smallskip

\emph{Compact sets:} If $X$ is net-compact, let $\mathcal U$ be an
ultrafilter on~$X$. Proposition~\ref{tfb-net} yields a net
$(x_\alpha)$ in $X$ with $\mathcal U=[x_\alpha]$. Use the
net-compactness of $X$ to find a quasi-subnet $(y_\beta)$ of
$(x_\alpha)$ such that $y_\beta\to x$ for some~$x$. It follows that
$\mathcal U=[x_\alpha]\subseteq[y_\beta]\to x$. Maximality of
$\mathcal U$ then implies $\mathcal U=[y_\beta]\to x$, so $X$ is
filter-compact. Conversely, suppose that $X$ is filter-compact. Let
$(x_\alpha)$ be a net in~$X$. Then $[x_\alpha]$ is a filter, so it is
contained in some ultrafilter $\mathcal U$. Then $\mathcal U\to x$ for
some~$x$. Now find a net $(y_\beta)$ with $\mathcal U=[y_\beta]$. We
clearly have $y_\beta\to x$, and it follows from
$[x_\alpha]\subseteq[y_\beta]$ that $(y_\beta)$ is a quasi-subnet of
$(x_\alpha)$. Therefore, $X$ is net-compact.

It is straightforward to verify that the concepts of \emph{local
  compactness} agree for nets and filters.

\smallskip

\emph{Bounded sets:} Let $X$ be a convergence vector space and
$A\subseteq X$. Suppose that $A$ is net-bounded. Since
$\mathcal N_0\to 0$ in~$\mathbb K$, we can find a net $(r_\alpha)$ in
$\mathbb K$ such that $[r_\alpha]=\mathcal N_0$ and $r_\alpha\to
0$. It follows that the net $(r_\alpha a)_{(\alpha,a)}$ converges to
zero in~$X$. It is easy to verify
$\mathcal N_0A=\langle r_\alpha a\rangle$, so
$[\mathcal N_0A]=[r_\alpha a]\to 0$. This means that $A$ is
filter-bounded. Conversely, suppose $A$ is filter-bounded and let
$r_\alpha\to 0$ in~$\mathbb K$. Then $[r_\alpha]\to 0$ means
$\mathcal N_0\subseteq[r_\alpha]$. It follows that every
$U\in\mathcal N_0$ contains a tail of $(r_\alpha)$ and, therefore,
$UA$ contains a tail of $(r_\alpha a)$. This gives
$\mathcal N_0A\subseteq\langle r_\alpha a\rangle$, so
$[\mathcal N_0A]\subseteq[r_\alpha a]$. Now $[\mathcal N_0A]\to 0$
implies $[r_\alpha a]\to 0$, hence $r_\alpha a\to 0$, showing that $A$
is net-bounded.

It is an easy exercise that a convergence vector space is
\emph{net-locally bounded} iff it is \emph{filter-locally bounded}.

\smallskip

\emph{Mackey modification:} Consider a convergence vector space
$(X,\to)$. We claim that its net-Mackey modification agrees with its
filter-Mackey modification in the sense that $x_\alpha\goesm x$ iff
$[x_\alpha]\goesm x$. Since both the original structure and its Mackey
modification are linear, we may assume that $x=0$.

Suppose that $x_\alpha\goesm 0$. Then there exists a bounded set $B$
such that for every $\varepsilon>0$ the set $\varepsilon B$ contains a
tail of $(x_\alpha)$. It follows that $\varepsilon B\in[x_\alpha]$.
Since $\varepsilon$ was arbitrary, we conclude that
$\mathcal N_0B\subseteq[x_\alpha]$ and, therefore,
$[x_\alpha]\goesm 0$.  For the converse we assume
$[x_\alpha]\goesm 0$. Find a bounded set $B$ such that
$\mathcal N_0B\subseteq[x_\alpha]$. If $\varepsilon>0$ is arbitrary,
$\varepsilon B\in\mathcal N_0 B$, and hence $\varepsilon B$ contains a
tail of $(x_\alpha)$. This yields $x_\alpha\goesm 0$.

In view of this equivalence, Proposition~\ref{Mack-bdd} implies
Proposition~3.7.17 in~\cite{Beattie:02}.

\subsection*{First countability.} The rest of this section is devoted
to translations of the first countability condition into the
language of nets.

\begin{proposition}
  A convergence space is first countable at $x$ iff for every net 
  $(x_\alpha)_{\alpha\in\Lambda}$ which converges to $x$ there exists
  a net $(y_\gamma)_{\gamma\in\Gamma}$ such that
  $(x_\alpha)\preceq(y_\gamma)$, $y_\gamma\to x$, and $\Gamma$ admits
  a countable co-final subset. 
\end{proposition}

\begin{proof}
  Suppose that $X$ is first countable and let
  $(x_\alpha)_{\alpha\in\Lambda}\to x$.  Find a filter
  $\mathcal G\subseteq[x_\alpha]$ such that $\mathcal G\to x$ and
  $\mathcal G$ has a countable base $\{B_n\}$. Replacing
  $B_n$ with $B_1\cap\dots\cap B_n$, we may assume
  that $(B_n)$ is nested and decreasing. Define
  $(y_\gamma)_{\gamma\in\Gamma}$ as in the proof of
  Proposition~\ref{tfb-net} by putting
  \begin{math}
    \Gamma=\bigl\{(n,b)\mid n\in\mathbb N,\ b\in B_n\bigr\},
  \end{math}
  pre-ordering it by the first component, and letting $y_{(n,b)}=b$. Then
  $\langle y_\gamma\rangle_{\gamma\in\Gamma}=\{B_n\}_{n\in\mathbb
    N}$ and $[y_\gamma]=\mathcal G$. It follows that
    $(x_\alpha)\preceq(y_\gamma)$ and $y_\gamma\to x$. For every
    $n\in\mathbb N$, pick any $b_n\in B_n$; the set
  \begin{math}
    \Gamma_0=\bigl\{(n,b_n)\mid n\in\mathbb N\bigr\}
  \end{math}
  is a countable co-final subset of~$\Gamma$.
  
  To prove the converse, suppose that $\mathcal F\to x$. Find
  $(x_\alpha)_{\alpha\in\Lambda}$ such that $\mathcal
  F=[x_\alpha]$. Then there exists a net $(y_\gamma)_{\gamma\in\Gamma}$ such that
  $(x_\alpha)\preceq(y_\gamma)$, $y_\gamma\to x$, and $\Gamma$ admits
  a countable co-final subset $\{\gamma_n\}$. Let $\mathcal
  G=[y_\gamma]$ and, for every~$n$, put
  $B_n=\{y_\gamma\}_{\gamma\ge\gamma_n}$. It follows that $\mathcal
  G\to x$, $\mathcal G\subseteq\mathcal F$, and $\{B_n\}$ is a
  countable base of~$\mathcal G$.
\end{proof}

\begin{proposition}\label{fcount-forw}
  Let $X$ be a convergence space. If $X$ is first countable at $x$ and
  $(x_\alpha)_{\alpha\in\Lambda}\to x$ then there exists an
  increasing sequence of indices $(\alpha_n)$ in $\Lambda$ such that
  $x_{\alpha_n}\to x$. Moreover, if $\beta_n\ge\alpha_n$ for every $n$
  then $x_{\beta_n}\to x$.
\end{proposition}

\begin{proof}
  By first countability, there exists a filter
  $\mathcal G\subseteq[x_\alpha]$ such that $\mathcal G\to x$ and
  $\mathcal G$ has a countable base, say, $\{B_n\}$. Again, we may
  assume that $(B_n)$ is nested and decreasing. Inductively, we find an
  increasing sequence $(\alpha_n)$ in $\Lambda$ such that
  $\{x_\alpha\}_{\alpha\ge\alpha_n}\subseteq B_n$ for every~$n$. It
  follows that $\{x_{\alpha_k}\}_{k\ge n}$ is contained in $B_n$ and,
  therefore, $B_n$ is in the tail filter of the sequence
  $(x_{\alpha_n})$. This yields $\mathcal G\subseteq[x_{\alpha_n}]$,
  so that $[x_{\alpha_n}]\to x$; that is, $x_{\alpha_n}\to x$.

  Similarly, if $\beta_n\ge\alpha_n$ for every $n$ then
  $\{x_{\beta_k}\}_{k\ge n}$ is contained in $B_n$, hence $\mathcal
  G\subseteq[x_{\beta_n}]$ and $x_{\beta_n}\to x$.
\end{proof}

This immediately yields that the closure of a set in a first countable
convergence space agrees with its sequential closure
(Proposition~1.6.4 in~\cite{Beattie:02}).

\begin{definition}
  Let $(x_\alpha)_{\alpha\in\Lambda}$ be a net and $(\alpha_n)$ an
  increasing sequence of indices in~$\Lambda$. Let
  \begin{displaymath}
  \Gamma=\{(n,\alpha)\mid n\in\mathbb N,\alpha\in\Lambda,\ \alpha\ge\alpha_n\}
  \end{displaymath}
  be pre-ordered by the first component. The net
  $y_{(n,\alpha)}=x_\alpha$ indexed by $\Gamma$ is called a
  \term{matryoshka} of $(x_\alpha)$.
\end{definition}

\begin{remark}\label{fold-tails}
  In the preceding definition, let
  $\gamma_0=(n_0,\alpha_0)\in\Gamma$. It can be easily verified that
  $\{y_\gamma\}_{\gamma\ge\gamma_0}$ equals
  $\{x_\alpha\}_{\alpha\ge\alpha_{n_0}}$. That is, the tail sets of a
  matryoshka are precisely the tail sets of the original net
  corresponding to the $(\alpha_n)$.
 \end{remark}

\begin{proposition}
  A convergence space is first countable at $x$ iff every net
  $(x_\alpha)$ that converges to $x$ admits a matryoshka that
  converges to~$x$.
\end{proposition}

\begin{proof}
  First we assume $X$ is first countable at $x$ and let
  $x_\alpha\to x$. Take $\mathcal G$, $(\alpha_n)$ and $(B_n)$ be as
  in the proof of Proposition~\ref{fcount-forw}, and let
  $(y_\gamma)_{\gamma\in\Gamma}$ be the matryoshka of $(x_\alpha)$
  corresponding to $(\alpha_n)$. By Remark~\ref{fold-tails}, each
  $B_n$ contains a tail of $(y_\gamma)$, so
  $\mathcal G\subseteq[y_\gamma]$. It follows that $y_\gamma\to x$.

  To prove the converse, suppose that $\mathcal F\to x$. Find a net
  $(x_\alpha)$ such that $\mathcal F=[x_\alpha]$. By assumption, there
  exists an increasing sequence $(\alpha_n)$ in $\Lambda$ such that
  the corresponding matryoshka $(y_\gamma)_{\gamma\in\Gamma}$ converges
  to~$x$. Then defining $\mathcal G=[y_\gamma]$ implies $\mathcal G\to x$. It
  follows from Remark~\ref{fold-tails} that
  $\mathcal G\subseteq\mathcal F$.
\end{proof}

\section{Mixings and (pre)topological structures}
\label{sec:mix}

In this section we present an equivalent reformulation of Axiom (N3)
based on the concept of \emph{mixing}. We then use this concept to
characterize \emph{pretopological} convergence structures in terms of nets.

Let $X$ be a set, and let $(x_\alpha)$, $(y_\alpha)$ and $(z_\alpha)$ be
three nets in $X$ with the same index sets. We say that $(z_\alpha)$
is a \term{braiding} of $(x_\alpha)$ and $(y_\alpha)$ if
$z_\alpha\in\{x_\alpha,y_\alpha\}$ for every~$\alpha$. In this
terminology, Axiom (N3) effectively says that if $x_\alpha\to x$ and
$y_\alpha\to x$ then every braiding of the two nets also converges to~$x$.

Next, we define a closely related concept of \term{mixing}.
Given two nets, $(x_\alpha)_{\alpha\in A}$ and $(y_\beta)_{\beta\in
  B}$, we construct a new net $(u_\gamma)_{\gamma\in\Gamma}$ which, in
a certain sense, combines $(x_\alpha)$ and $(y_\beta)$ in a single
net. Set
\begin{displaymath}
  \Gamma=\bigl\{(\alpha,\beta,z)
  \mid \alpha\in A,\ \beta\in B,\ z\in\{x_\alpha,y_\beta\}\bigr\}
\end{displaymath}
and give it the following pre-order:
\begin{math}
  (\alpha_1,\beta_1,z_1)\le(\alpha_2,\beta_2,z_2)
\end{math}
whenever $\alpha_1\le\alpha_2$ and $\beta_1\le\beta_2$. It is clear
that $\Gamma$ is a directed set. We now put
$u_{(\alpha,\beta,z)}=z$. The resulting net is called the \term{mixing}
of $(x_\alpha)$ and $(y_\beta)$.

The following two lemmas characterize the relationships between
braiding and mixing. 

\begin{lemma}
  Every braiding of $(x_\alpha)_{\alpha\in A}$ and
  $(y_\alpha)_{\alpha\in A}$ is a quasi-subnet of their mixing.
\end{lemma}

\begin{proof}
  Let $(z_\alpha)_{\alpha\in A}$ be a braiding of  $(x_\alpha)_{\alpha\in A}$ and
  $(y_\alpha)_{\alpha\in A}$, and let  $(u_\gamma)_{\gamma\in\Gamma}$ be
  the mixing of  $(x_\alpha)_{\alpha\in A}$ and
  $(y_\alpha)_{\alpha\in A}$. For every $\alpha\in A$, the triple
  $\gamma=(\alpha,\alpha,z_\alpha)$ belongs to $\Gamma$ and
  $u_\gamma=z_\alpha$. Put
  \begin{displaymath}
    \Gamma_0=\{(\alpha,\alpha,z_\alpha)\mid \alpha\in A\}.
  \end{displaymath}
  It is easy to see that $\Gamma_0$ is a directed subset
  of~$\Gamma$. Moreover, $\Gamma_0$ is co-final in~$\Gamma$, so
  $(u_\gamma)_{\gamma\in\Gamma_0}$ is a quasi-subnet of
  $(u_\gamma)_{\gamma\in\Gamma}$. Since the map
  $\alpha\mapsto(\alpha,\alpha,z_\alpha)$ is an order isomorphism
  between $A$ and~$\Gamma_0$, the nets $(z_\alpha)_{\alpha\in A}$ and
  $(u_\gamma)_{\gamma\in\Gamma_0}$ are the same net up to this
  isomorphism between the index sets.
\end{proof}

\begin{lemma}
  The mixing of two nets is a braiding of a pair of nets that are
  strongly tail equivalent to the original ones.
\end{lemma}

\begin{proof}
  Let $(u_\gamma)_{\gamma\in\Gamma}$ be the mixing of
  $(x_\alpha)_{\alpha\in A}$ and $(y_\beta)_{\beta\in B}$. Define two
  nets $(x'_\gamma)_{\gamma\in\Gamma}$ and
  $(y'_\gamma)_{\gamma\in\Gamma}$ as follows: for
  $\gamma=(\alpha,\beta,z)\in\Gamma$, we put $x'_\gamma=x_\alpha$ and
  $y'_\gamma=y_\beta$. By the definition of~$\Gamma$, we have
  \begin{displaymath}
    u_\gamma=z\in\{x_\alpha,y_\beta\}=\{x'_\gamma,y'_\gamma\},
  \end{displaymath}
  hence $(u_\gamma)$ is a braiding of $(x'_\gamma)$ and
  $(y'_\gamma)$. On the other hand, it is easy to see that
  $(x'_\gamma)_{\gamma\in\Gamma}$  and $(y'_\gamma)_{\gamma\in\Gamma}$
  are strongly tail equivalent to $(x_\alpha)_{\alpha\in A}$ and
  $(y_\alpha)_{\beta\in B}$, respectively.
\end{proof}

Combining the two lemmas, we immediately get the following.

\begin{theorem}
  Axiom (N3) in the definition of net convergent structures may be
  replaced with
  \begin{enumerate}
  \item[(N3*)] If two nets converge to $x$ then their mixing converges
    to~$x$.
  \end{enumerate}
\end{theorem}

We now extend the definition of mixing to an arbitrary set of
nets. For these purposes, it is useful to introduce some notation.

\begin{definition}\label{mixing}
  Let $\Sigma=\{S_j\mid j\in J\}$ be a set of nets indexed by a
  set~$J$; so $S_j$ denotes a net for each $j\in J$. Note that we do
  not assume any structure on~$J$. Each net in $\Sigma$ is associated
  with a directed pre-ordered index set~$A_j$; i.e.,
  $S_j=(x_{\alpha^{(j)}})_{\alpha^{(j)}\in A_j}$ for each $j\in
  J$.  Define
  \begin{math}
  \Gamma=\Bigl(\prod_{j\in J}A_j\Bigr)\times J
  \end{math}
  and equip it with the product pre-order on the first component: for
  $\gamma_1=(\bar\alpha_1,j_1)$ and $\gamma_2=(\bar\alpha_2,j_2)$, we
  have $\gamma_1\le\gamma_2$ whenever
  $\alpha^{(j)}_1\le\alpha^{(j)}_2$ for all $j\in J$. It is easy to
  see that $\Gamma$ is directed. Now put $\gamma=(\bar\alpha,j)$ and
  define $y_\gamma=x_{\bar{\alpha}(j)}$; this net
  $M=(y_\gamma)_{\gamma\in\Gamma}$ is called the \term{mixing}
  of~$\Sigma$.  When we do not wish to draw attention to the index
  sets of the nets in~$\Sigma$, we will suppress the reference to the
  index set $J$ and just write $S\in \Sigma$ to denote a net from the
  family.
\end{definition}

It is easy to see this mixing net $M$ visits every term of every net
$S$ in~$\Sigma$. Moreover, it visits them in the
same order. Also note the definition of mixing for arbitrary sets of
nets agrees with the mixing of two nets from before.

\begin{lemma}
  Let $\Sigma$ be a family of nets in $X$ and $M$ the mixing
  of~$\Sigma$. Then each of the nets that make up $\Sigma$ is a
  quasi-subnet of~$M$.
\end{lemma}

\begin{proof}
  Using the notation above, fix $j_0\in J$ and let
  $\Gamma_0\subseteq\Gamma$ be the ``$j_0$-th column'' of~$\Gamma$:
  \begin{math}
     \Gamma_0=\Bigl(\prod_{j\in J}A_j\Bigr)\times\{j_0\}.
  \end{math}
  It is easy to see that $\Gamma_0$ is a directed co-final subset
  of~$\Gamma$, hence $(y_\gamma)_{\gamma\in\Gamma_0}$ is a
  quasi-subnet of~$M$. On the other hand,
  $(y_\gamma)_{\gamma\in\Gamma_0}$ has the same terms as $S_{j_0}$ and
  the two nets are strongly tail equivalent.
\end{proof}

\begin{lemma}
  Let $\Sigma$ be a family of nets in $X$ and $M$ the mixing
  of~$\Sigma$. Then the tail filter of $M$ is the intersection of the
  tail filters of the nets that make up~$\Sigma$.
\end{lemma}

\begin{proof}
  Recall that we write $[M]$ for the tail filter
  of~$M$. Thus, we need to show that $[M]=\bigcap_{S\in \Sigma}[S]$.

  Let $A\in[M]$. Then $A$ contains a tail of~$M$, hence a tail of
  every $S\in\Sigma$, because each $S$ is a quasi-subnet of $M$ by the
  Lemma above.

  Conversely, suppose $A\in\bigcap_{S\in \Sigma}[S]$. For every
  $S\in \Sigma$, $A$ contains a tail of~$S$. It is a routine
  verification that the mixing of these tails is a tail of $M$ that
  is still contained in~$A$.
\end{proof}

Recall that if $X$ is a (filter) convergence space and $x\in X$, the
neighborhood filter $\mathcal U_x$ of $x$ is defined as the
intersection of all filters that converge to~$x$. A convergence
structure is \term{pretopological} if $\mathcal U_x\to x$ for
every~$x$. This terminology is motivated by the situation in
topological convergence structures: the definition of $\mathcal U_x$
agrees with the usual neighborhood filter~$\mathcal N_x$, which
converges to~$x$. We are now ready to show that extending (N3*) from
two nets to an arbitrary collection of nets characterizes
pretopological convergence structures in terms of nets.

\begin{theorem}\label{pretop}
  A convergence space $X$ is pretopological iff for every $x\in X$ and
  every family of nets converging to~$x$, the mixing of the family
  also converges to~$x$.
\end{theorem}

\begin{proof}
  Suppose that the convergence structure is pretopological. Let
  $\Sigma$ be a collection of nets converging to $x$ where $M$ is the
  mixing of~$\Sigma$. For every $S\in\Sigma$ we have $S\to x$, hence
  $[S]\to x$. It follows from
  \begin{math}
    [M]=\bigcap_{S\in\Sigma}[S]\supseteq\mathcal U_x\to x
  \end{math}   
  that $[M]\to x$ and, therefore, $M\to x$. 

  To prove the converse, note that for each filter $\mathcal F$ that
  converges to $x$ we can find a net $S$ such that $S\to x$ and
  $\mathcal F=[S]$. Denote the resulting collection of nets
  by~$\Sigma$ and let $M$ be the mixing of~$\Sigma$. Then $M\to x$ and
  we get 
  \begin{math}
    \mathcal U_x=\bigcap_{S\in\Sigma}[S]=[M]\to x.
  \end{math}
\end{proof}

We now characterize topological convergence structures in similar
terms. There is extensive literature on this subject:
\cite[p.~74]{Kelley:76}, \cite{Aarnes:72}, \cite[15.10]{Schechter:97},
and \cite{Beattie:02} are some examples. Of note, Proposition~1.3.21
of~\cite{Beattie:02} shows that a convergence structure is topological
iff it is pretopological and the closure of every set is
closed. Compare this to \cite[p.~69]{Kelley:76}, which proves that a
convergence structure is topological when it satisfies a certain
\emph{iterated limit property}. In the case of sequences, this
essentially means the limit of a double sequence $(x_{n,m})$ agrees
with its iterated limit; i.e., $\lim x_{m,n}=\lim_m\lim x_{m,n}$
whenever the latter limit exists. For nets, the iterated limit
property is closely related to the concept of mixing.

As in Definition~\ref{mixing}, we let $\Sigma=\{S_j\mid j\in J\}$
denote a set of nets and emphasize the index sets $A_j$ for each
net~$S_j$; that is, $S_j=(x_{\alpha^{(j)}})_{\alpha^{(j)}\in A_j}$ for
each $j\in J$. The only difference now is that we make the additional
assumption that $J$ is a directed pre-ordered set. This allows us to
view $\Sigma$ as a net $(S_j)_{j\in J}$ of nets. Take
\begin{math}
  \Delta=\Bigl(\prod_{j\in J}A_j\Bigr)\times J
\end{math}
and equip it with the product pre-order: for
$\gamma_1=(\bar\alpha_1,j_1)$ and $\gamma_2=(\bar\alpha_2,j_2)$ we
declare $\gamma_1\le\gamma_2$ whenever
$\alpha^{(j)}_1\le\alpha^{(j)}_2$ for all $j\in J$ and $j_1\le j_2$.
It is easy to see that this makes $\Delta$ into a directed set. For
$\rho=(\bar\alpha,j)\in \Delta$, put $r_\rho=x_{\bar{\alpha}(j)}$. The
resulting net $R=(r_\rho)_{\rho\in\Delta}$ is called the
\term{reaction} of $J$ with~$\Sigma$.

Note that as a set, $\Delta$ coincides with $\Gamma$ in the definition
of mixing; however, the two differ as directed sets: the only difference is
$\Delta$ can ``see'' an order on~$J$. Similarly, $R$ and the mixing
of~$\Sigma$ agree as functions but are distinct as nets: $R$ is a
quasi-subnet of~$M$.

We say that a convergence structure has the \term{iterated limit
property} if $\lim R=\lim_{j\in J}\lim S_j$, whenever the limit in
the right hand side exists, for every net $\Sigma$ of nets in~$X$.

The following result is essentially contained in
\cite[15.10]{Schechter:97}, but we present a shorter proof.

\begin{theorem}
  A convergence structure $X$ is topological iff it satisfies the
  iterated limit property.
\end{theorem}

\begin{proof}
  The forward implication is proved in~\cite[p.~69]{Kelley:76}. For
  the converse, suppose that $X$ has the iterated limit property. It
  suffices to verify that $X$ is pretopological and the closure of
  every set is closed. The latter condition is straightforward.  To
  prove that $X$ is pretopological we use Theorem~\ref{pretop}. Let
  $\Sigma=\{S_j\mid j\in J\}$ be as in Definition~\ref{mixing}, and let
  $M$ be the mixing of~$\Sigma$. Assume that $S_j\to x$ for every
  $j\in J$. We will show $M\to x$. Give $J$ the trivial pre-order by declaring
  $j_1\le j_2$ for every $j_1,j_2\in J$, and let $R$ be the reaction of $J$
  with~$\Sigma$. Then $R\to x$ by the iterated limit property. However, it
  is easy to see that under our pre-ordering of~$J$, $M$ and $R$ agree
  as nets. It immediately follows that $M\to x$.
\end{proof}

\section{Continuous convergence}
\label{sec:cc}

In this section we apply net convergence theory to reprove some
results about continuous convergence from~\cite{Beattie:02}.  While
the results themselves are nothing new, our arguments highlight the
simplicity of the net language for functional analysis.

Let $X$ and $Y$ be two convergence spaces, and use $C(X,Y)$ to denote
the set of all continuous functions from $X$ to~$Y$.  \term{Continuous
  convergence} on $C(X,Y)$ is defined as follows: $f_\alpha\goesc f$
if $f_\alpha(x_\beta)\to f(x)$ in $Y$ whenever $x_\beta\to x$
in~$X$. Note that we treat
$\bigl(f_\alpha(x_\beta)\bigr)_{(\alpha,\beta)}$ as a net over the
product of the index sets. Note also that this net need not be
admissible.  It is easy to verify that continuous convergence is a net
convergence structure. One should compare this with the definition in
terms of filters: for a filter $\Phi$ on $C(X,Y)$, $\Phi\goesc f$ if
$\Phi(\mathcal F)\to f(x)$ whenever $\mathcal F\to x$; here
$\mathcal F$ is a filter on $X$ and $\Phi(\mathcal F)$ is the filter
generated by the sets of the form $F(A)$ where $F\in\Phi$ and
$A\in\mathcal F$.

These definitions are equivalent in the sense that $\Phi\goesc f$ iff
$\Phi=[f_\alpha]$ for some net $(f_\alpha)$ with $f_\alpha\goesc
f$. This is an easy corollary of the fact that
$\Phi(\mathcal F)=\bigl[f_\alpha(x_\beta)\bigr]$ whenever
$\Phi=[f_\alpha]$ and $\mathcal F=[x_\beta]$, where $(f_\alpha)$ is a
net in $C(X,Y)$ and $(x_\beta)$ a net in~$X$.

We will write $C_c(X,Y)$ for the set $C(X,Y)$ equipped with the
continuous convergence structure. We will write $C_c(X)$ in place of
$C_c(X,\mathbb K)$ where the scalar field is always assumed to carry
the standard convergence. It is easy to see that $C_c(X)$ is a
convergence vector space.

The next few results remain valid, with similar proofs, if we replace
$C_c(X)$ with $C_c(X,Y)$ where $Y$ is a Hausdorff complete topological
vector space. We use $C_c(X)$ rather than $C_c(X,Y)$ for simplicity.

\begin{proposition}\label{Cc-complete}
  For every convergence space~$X$, $C_c(X)$ is complete.
\end{proposition}

\begin{proof}
  The proof is a standard approximation argument.  Let
  $(f_\alpha)_{\alpha\in\Lambda}$ be Cauchy net in $C_c(X)$. In
  particular, for every $x\in X$, the double net
  $\bigl(f_\alpha(x)-f_\beta(x))_{(\alpha,\beta)}$ converges to zero
  in~$\mathbb K$. It follows that the net $\bigl(f_\alpha(x)\bigr)$ is
  Cauchy and, therefore, converges to a unique limit in the usual
  convergence of~$\mathbb K$. Write $f(x)=\lim_\alpha f_\alpha(x)$. It
  is left to show that $f$ is continuous and $f_\alpha\goesc f$.

  Let $(x_\gamma)_{\gamma\in\Gamma}$ be a net in $X$ with $x_\gamma\to
  x$. It follows from $f_\alpha-f_\beta\goesc 0$ that
  $f_\alpha(x_\gamma)-f_\beta(x_\gamma)\to 0$ in~$\mathbb K$, where the
  expression is a net over $\Lambda\times\Lambda\times\Gamma$. Fix
  $\varepsilon>0$. There exist $\alpha_0$, $\beta_0$, and $\gamma_0$
  such that
  \begin{displaymath}
    \bigabs{f_\alpha(x_\gamma)-f_\beta(x_\gamma)}<\varepsilon
  \end{displaymath}
  whenever $\alpha\ge\alpha_0$, $\beta\ge\beta_0$, and
  $\gamma\ge\gamma_0$.

  Fix $\beta_1\ge\beta_0$ such that
  \begin{displaymath}
    \bigabs{f_{\beta_1}(x)-f(x)}<\varepsilon;
  \end{displaymath}
  this can be done by the
  definition of $f(x)$. Since $f_{\beta_1}$ is continuous, we can
  find $\gamma_1\ge\gamma_0$ such that
  \begin{displaymath}
    \bigabs{f_{\beta_1}(x_\gamma)-f_{\beta_1}(x)}<\varepsilon
  \end{displaymath}
  whenever $\gamma\ge\gamma_1$. Fix $\gamma\ge\gamma_1$. Using the
  definition of $f$ again, we find some $\alpha_1\ge\alpha_0$
  such that
  \begin{displaymath}
    \bigabs{f_{\alpha_1}(x_\gamma)-f(x_\gamma)}<\varepsilon.
  \end{displaymath}
  Combining these inequalities, we get
  \begin{multline*}
    \bigabs{f(x_\gamma)-f(x)}
    \le\bigabs{f(x_\gamma)-f_{\alpha_1}(x_\gamma)}
      +\bigabs{f_{\alpha_1}(x_\gamma)-f_{\beta_1}(x_\gamma)}\\
      +\bigabs{f_{\beta_1}(x_\gamma)-f_{\beta_1}(x)}
      +\bigabs{f_{\beta_1}(x)-f(x)}
      <4\varepsilon.
  \end{multline*}
  It follows that $f(x_\gamma)\to f(x)$, hence $f$ is
  continuous. Furthermore, it follows from the first three
  inequalities that for every $\alpha\ge\alpha_0$ and every
  $\gamma\ge\gamma_1$ we have
  \begin{displaymath}
    \bigabs{f_\alpha(x_\gamma)-f(x)}
    \le\bigabs{f_\alpha(x_\gamma)-f_{\beta_1}(x_\gamma)}
    +\bigabs{f_{\beta_1}(x_\gamma)-f_{\beta_1}(x)}
    +\bigabs{f_{\beta_1}(x)-f(x)}
    <3\varepsilon,
  \end{displaymath}
  which yields $f_\alpha\goesc f$.
\end{proof}

We now define \term{uniform convergence on compacta (ucc)}.  Let $X$
be a convergence space and $f\in C(X)$. For a net $(f_\alpha)$ in
$C(X)$, we write $f_\alpha\goesucc f$ if $f_\alpha$ converges to $f$
uniformly on every compact subset of~$X$.  Note that ucc convergence
on $C(X)$ is topological and locally convex; this topology has a base
of zero neighborhoods given by the sets
\begin{displaymath}
  V_{A,\varepsilon}=\bigl\{f\in C(X)\mid \bigabs{f(x)}<\varepsilon
  \mbox{ for all }x\in A\bigr\},
\end{displaymath}
where $\varepsilon$ is a positive real and $A$ is a compact subset of~$X$.

\begin{proposition}
  Let $X$ be a convergence space. If $f_\alpha\goesc f$ then
  $f_\alpha\goesucc f$ in $C(X)$.
\end{proposition}

\begin{proof}
  By linearity we may assume $f=0$. For the sake of contradiction, let
  $(f_\alpha)_{\alpha\in\Lambda}$ be a net in $C(X)$ such that
  $f_\alpha\goesc 0$ but $f_\alpha$ fails to converge to $0$ uniformly
  on some compact subset $A$ of~$X$. It follows that there exists
  $\varepsilon>0$ such that for every $\alpha\in\Lambda$ there exists
  $x_\alpha\in A$ and $\beta=\beta(\alpha)$ in $\Lambda$ with
  $\beta(\alpha)\ge\alpha$ such that
  $\bigabs{f_{\beta(\alpha)}(x_\alpha)}>\varepsilon$. Then
  $g_\alpha=f_{\beta(\alpha)}$ defines a quasi-subnet of
  $(f_\alpha)$, hence $g_\alpha\goesc 0$. Employing the compactness of
  $A$ and passing to a further quasi-subnet, we may assume
  $(x_\alpha)$ is convergent. From this we obtain
  $g_\alpha(x_\alpha)\to 0$, which contradicts
  $\bigabs{g_\alpha(x_\alpha)}=\bigabs{f_{\beta(\alpha)}(x_\alpha)}>\varepsilon$
  for all~$\alpha$.
\end{proof}

\begin{corollary}\label{loc-conv-ucc}
  Let $X$ be a locally compact convergence space. Then the continuous
  and the ucc convergence structures on $C(X)$ agree. In particular,
  $C_c(X)$ is a locally convex topological vector space.
\end{corollary}

\begin{proof}
  We already know that $f_\alpha\goesc f$ implies $f_\alpha\goesucc
  f$. Suppose that $f_\alpha\goesucc f$ in $C_c(X)$ and $x_\beta\to x$
  in~$X$. Fix $\varepsilon>0$. Since $X$ is locally compact, there
  exists a compact set $A$ in $X$ and an index $\beta_0$ such that
  $x_\beta\in A$ for all $\beta\ge\beta_0$. The continuity of $f$ yields
  some $\beta_1$ such that
  $\bigabs{f(x_\beta)-f(x)}<\varepsilon$ whenever
  $\beta\ge\beta_1$. Since $f_\alpha\goesucc f$, there exists an index
  $\alpha_0$ such that $\bigabs{f_\alpha(z)-f(z)}<\varepsilon$ for all
  $\alpha\ge\alpha_0$ and all $z\in A$. It follows that for all
  $\alpha\ge\alpha_0$ and all $\beta$ such that $\beta\ge\beta_0$ and
  $\beta\ge\beta_1$, we have
  \begin{displaymath}
    \bigabs{f_\alpha(x_\beta)-f(x)}
    \le\bigabs{f_\alpha(x_\beta)-f(x_\beta)}+\bigabs{f(x_\beta)-f(x)}
    <2\varepsilon
  \end{displaymath}
  implying $f_\alpha(x_\beta)\to f(x)$.
\end{proof}

Let $X$ be a convergence vector space. Following~\cite{Beattie:02}, we
write $\mathcal L_cX$ for the space of all continuous linear
functionals on~$X$ equipped with the continuous convergence
structure. Note that $\mathcal L_cX$ is a closed subspace of $C_c(X)$.
Then Proposition~\ref{Cc-complete} implies $\mathcal L_cX$ is
complete. One may view $\mathcal L_cX$ as the dual space of $X$ in the
category of convergence vector spaces.

Let $X$ be a topological vector space. As a set, $\mathcal L_cX$
agrees with the topological dual $X^*$ of~$X$.  The next proposition
compares the continuous convergence on $\mathcal L_cX$ with the
weak*-convergence, i.e., convergence in the $\sigma(X^*,X)$-topology.

\begin{proposition}\label{c-ws-pol}
  Let $X$ be a topological vector space. Then $f_\alpha\goesc f$ in
  $X^*$ iff $f_\alpha\goesws f$ and a tail of $(f_\alpha)$ is
  contained in the polar of a zero neighborhood in~$X$.
\end{proposition}

\begin{proof}
  Suppose $f_\alpha\goesc f$. It is immediate that
  $f_\alpha\goesws f$. Let $\mathcal N_0^X$ be the filter of all zero
  neighborhoods in~$X$.  Since $\mathcal N_0^X\to 0$, we can find a net
  $(x_\beta)$ in $X$ such that $x_\beta\to 0$ and
  $[x_\beta]=\mathcal N_0^X$. It follows that $f_\alpha(x_\beta)\to 0$.
  Then there exists $\alpha_0$ and $\beta_0$ such that
  $\bigabs{f_\alpha(x_\beta)}\le 1$ whenever $\alpha\ge\alpha_0$ and
  $\beta\ge\beta_0$. Set $V=\{x_\beta\}_{\beta\ge\beta_0}$ and observe
  $V\in\mathcal N_0^X$ and
  $\{f_\alpha\}_{\alpha\ge\alpha_0}\subseteq V^\circ$.

  Conversely, suppose that $f_\alpha\goesws f$ and a tail
  $\{f_\alpha\}_{\alpha\ge\alpha_0}$ is contained in $V^\circ$ for
  some $V\in\mathcal N_0^X$. Suppose $x_\beta\to x$, and let
  $\varepsilon>0$ be fixed. Find $\alpha_1\ge\alpha_0$ such that
  $\abs{f_\alpha(x)-f(x)}<\varepsilon$ whenever
  $\alpha\ge\alpha_1$, and find $\beta_1$ such that
  $x_\beta-x\in\varepsilon V$ for all $\beta\ge\beta_1$. Now
  $\alpha\ge\alpha_1$ and $\beta\ge\beta_1$ imply
  \begin{displaymath}
    \bigabs{f_\alpha(x_\beta)-f(x)}
    \le\bigabs{f_\alpha(x_\beta-x)}+\bigabs{f_\alpha(x)-f(x)}
    <2\varepsilon.
  \end{displaymath}
  Hence, $f_\alpha(x_\beta)\to f(x)$.
\end{proof}

\begin{proposition}\label{TVS-loc-comp}
  Let $X$ be a topological vector space. The polar of every zero
  neighborhood in $X$ is compact in $\mathcal L_cX$. Consequently,
  $\mathcal L_cX$ is locally compact.
\end{proposition}

\begin{proof}
  Let $V\in\mathcal N_0^X$.  That $V^\circ$ is w*-compact is standard;
  see, e.g., page~139 in~\cite{Kelley:76}. Indeed, $V^\circ$ is
  w*-bounded and, since the w*-topology on $X^*$ is the restriction to
  $X^*$ of the product topology on $\mathbb K^X$, it is relatively
  w*-compact. Now $V^\circ$ is w*-closed implies it is w*-compact. It
  follows that every net $(f_\alpha)$ in $V^\circ$ has a w*-convergent
  quasi-subnet $g_\beta\goesws g$ for some $g\in
  V^\circ$. Proposition~\ref{c-ws-pol} gives $g_\beta\goesc g$,
  implying $V^\circ$ is compact in $\mathcal L_cX$.
  
  We have just shown that every continuously convergent net in
  $\mathcal L_cX$ has a tail inside a continuously compact set. An
  application of Proposition~\ref{c-ws-pol} gives $\mathcal L_cX$ is
  locally compact.
\end{proof}

\begin{corollary}\label{LcLc-lctvs}
  If $X$ is a topological vector space then $\mathcal L_c\mathcal
  L_cX$ is a locally convex topological vector space with the
  topology of uniform convergence on compact sets in $\mathcal L_cX$.
\end{corollary}

\begin{proof}
  By Proposition~\ref{TVS-loc-comp}, $\mathcal L_cX$ is locally
  compact. Corollary~\ref{loc-conv-ucc} implies that
  $C_c(\mathcal L_cX)$ is a locally convex topological vector space
  with the topology of uniform convergence on compact sets in
  $\mathcal L_cX$. The conclusion now follows from the fact that
  $\mathcal L_c\mathcal L_cX$ is a subspace of $C_c(\mathcal L_cX)$.
\end{proof}

\begin{proposition}\label{c-ucc-polar}
  Let $X$ be a topological vector space and $(\varphi_\alpha)$ a net in
  $\mathcal L_c\mathcal L_cX$. Then $\varphi_\alpha\goesc 0$ iff
  $(\varphi_\alpha)$ converges to zero uniformly on $V^\circ$ for every
  zero neighborhood $V$ in~$X$.
\end{proposition}

\begin{proof}
  If $\varphi_\alpha\goesc 0$ then it converges to zero uniformly on
  compact sets by Corollary~\ref{loc-conv-ucc}, and $V^\circ$ is
  compact by Proposition~\ref{TVS-loc-comp}. To prove the converse,
  suppose that $(\varphi_\alpha)$ converges to zero uniformly on
  polars of zero neighborhoods in~$X$. Let $(f_\beta)$ be a
  continuously convergent net in $\mathcal L_cX$, and let
  $\varepsilon>0$ be arbitrary. Proposition~\ref{c-ws-pol} guarantees
  that there exists a zero neighborhood $V$ in $X$ and an index
  $\beta_0$ such that $f_\beta\in V^\circ$ for all
  $\beta\ge\beta_0$. Since $(\varphi_\alpha)$ converges to zero
  uniformly on $V^\circ$, we can find $\alpha_0$ such that
  $\bigabs{\varphi_\alpha(f)}<\varepsilon$ for all $\alpha\ge\alpha_0$
  and $f\in V^\circ$. It follows that
  $\bigabs{\varphi_\alpha(f_\beta)}<\varepsilon$ for all
  $\alpha\ge\alpha_0$ and $\beta\ge\beta_0$.
\end{proof}

Let $X$ be a convergence vector space. For every $x\in X$, the
evaluation map $\hat x\colon\mathcal L_cX\to\mathbb K$ given by
$\hat x(f)=f(x)$ is continuous, hence
$\hat x\in\mathcal L_c\mathcal L_cX$. This map is clearly linear.  It
follows easily from the definition of continuous convergence that the
map $j\colon X\to\mathcal L_c\mathcal L_cX$ given by $j(x)=\hat x$ is
continuous. We say that $X$ is \term{reflexive} if $j$ is an
isomorphism. Since $\mathcal L_cX$ is complete for every convergence
vector space~$X$, we immediately get the following:

\begin{proposition}
  Every reflexive convergence vector space is complete.
\end{proposition}

\begin{proposition}\label{Lc-embeds-LcLcLc}
  Let $X$ be a convergence vector space. The map $f\mapsto \hat f$ is
  an isomorphic embedding of $\mathcal L_cX$ into $\mathcal
  L_c\mathcal L_c\mathcal L_cX$.
\end{proposition}

\begin{proof}
  We already know that this map is linear and continuous.  It is
  one-to-one: if $\hat f=0$ then $0=\hat f(\hat x)=f(x)$ for every
  $x\in X$, so that $f=0$. Finally, suppose that
  $\hat f_\alpha\goesc 0$ in $\mathcal L_c\mathcal L_c\mathcal
  L_cX$. Let $x_\beta\to 0$ in~$X$. Then $\hat x_\beta\goesc 0$ in
  $\mathcal L_c\mathcal L_cX$, so that
  $f_\alpha(x_\beta)=\hat f_\alpha(\hat x_\beta)\to 0$. Therefore,
  $f_\alpha\goesc 0$ in $\mathcal L_cX$.
\end{proof}

For every convergence vector space~$X$, one may view 
$\langle \mathcal L_cX,\mathcal L_c\mathcal L_cX\rangle$ as a dual
pair of vector spaces for every convergence vector space $X$ in the
sense of pp.~144--145 in~\cite{Aliprantis:06}. We will make use of the
classical Mackey-Arens Theorem (see Theorem~3.24 on p.~150
of~\cite{Aliprantis:06}).

\begin{theorem}[Mackey-Arens]
  Let $\langle X,Y\rangle$ be a dual pair of vector spaces and $\tau$
  a locally convex topology on~$X$. Then the dual of $(X,\tau)$ is
  $Y$ iff $\tau$ is the topology of uniform convergence on some
  collection of convex circled $\sigma(Y,X)$-compacts subsets of~$Y$.
\end{theorem}

\begin{theorem}\label{LcX-refl}
  If $X$ is a topological vector space then $\mathcal L_cX$ is reflexive.
\end{theorem}

\begin{proof}
  By Proposition~\ref{Lc-embeds-LcLcLc} it suffices to show that the
  map $f\mapsto \hat f$ is onto. Corollary~\ref{LcLc-lctvs} implies
  $\mathcal L_c\mathcal L_cX$ is a locally convex topological vector
  space; denote its topology by~$\tau^{**}$.  Consider the dual pair
  $\langle\mathcal L_c\mathcal L_cX,\mathcal L_cX\rangle$.  If
  $f_\alpha\goesc 0$ in $\mathcal L_cX$ then $\varphi(f_\alpha)\to 0$
  for every $\varphi\in\mathcal L_c\mathcal L_cX$, hence, $(f_\alpha)$
  converges to zero in
  $\sigma(\mathcal L_cX,\mathcal L_c\mathcal L_cX)$. It follows that
  the identity map on $\mathcal L_cX$ is continuous as a map from the
  continuous convergence structure to
  $\sigma(\mathcal L_cX,\mathcal L_c\mathcal L_cX)$.

  Let $V$ be a zero neighborhood in~$X$. Then $V^\circ$ is convex and
  circled. Furthermore, $V^\circ$ is compact in $\mathcal L_cX$ by
  Proposition~\ref{TVS-loc-comp}. By the preceding paragraph,
  $V^\circ$ is
  $\sigma(\mathcal L_cX,\mathcal L_c\mathcal L_cX)$-compact.  Let
  $\mathcal S$ be the set of all polars of zero neighborhoods
  in~$X$. Then $\mathcal S$ is a set of convex, circled, and
  $\sigma(\mathcal L_cX,\mathcal L_c\mathcal L_cX)$-compact subsets of
  $\mathcal L_cX$. By Proposition~\ref{c-ucc-polar}, $\tau^{**}$ is
  the topology of uniform convergence on sets in~$\mathcal S$. By
  the Mackey-Arens Theorem, the dual of
  $(\mathcal L_c\mathcal L_cX,\tau^{**})$ is $\mathcal L_cX$.

  Since the convergence on $\mathcal L_c\mathcal L_cX$ is topological,
  its continuous and topological dual agree as sets. It follows that
  $\mathcal L_c\mathcal L_c\mathcal L_cX$ and $\mathcal L_cX$ agree as
  sets. This means that the map $f\mapsto \hat f$ is onto.
\end{proof}

\begin{theorem}
  Let $(X,\tau)$ be a locally convex topological vector space. Then
   $\mathcal L_c\mathcal L_cX$ is the completion of~$X$. In
   particular, $X$ is reflexive iff it is complete.
\end{theorem}

\begin{proof}
  Let $j\colon X\to\mathcal L_c\mathcal L_cX$ be as before. Recall
  that $\mathcal L_cX$ and $X^*$ agree as sets. By
  Corollary~\ref{LcLc-lctvs}, $\mathcal L_c\mathcal L_cX$ is a locally
  convex topological vector space whose topology will be denoted
  by~$\tau^{**}$. We already know that $j$ is continuous as a map from
  $(X,\tau)$ to $(\mathcal L_c\mathcal L_cX,\tau^{**})$. It is
  one-to-one: if $\hat x=0$ then $f(x)=\hat x(f)=0$ for all
  $f\in X^*$, so that $x=0$.

  Next, we claim that $j$ is a homeomorphism. Suppose
  $\hat x_\alpha\goesc 0$ in $\mathcal L_c\mathcal L_cX$, but
  $x_\alpha\not\to 0$ in~$X$. Passing to a quasi-subnet, we find a
  convex circled neighborhood $V$ of zero in $X$ such that
  $x_\alpha\notin V$ for all~$\alpha$. By Hahn-Banach Theorem, for
  every~$\alpha$, we can find $f_\alpha\in X^*$ such that
  $f_\alpha(x_\alpha)\ge 1$ but $\bigabs{f_\alpha(x)}\le 1$ for all
  $x\in V$, i.e., $f_\alpha\in V^\circ$. Since $V^\circ$ is compact in
  $\mathcal L_cX$,
  by passing to
  further quasi-subnets of $(f_\alpha)$ and $(x_\alpha)$ we may
  assume $f_\alpha\goesc f$ for some~$f$. As $\hat
  x_\alpha\goesc 0$ in $\mathcal L_c\mathcal L_cX$, we have
  \begin{math}
    f_\alpha(x_\alpha)=\hat x_\alpha(f_\alpha)\to 0
  \end{math}
  which contradicts $f_\alpha(x_\alpha)\ge 1$ for all~$\alpha$.

  The previous paragraph shows $j(X)$ is a (linear topological)
  subspace of $\mathcal L_c\mathcal L_cX$. We claim that it is
  dense. Assume towards a contradiction that $j(X)$ is not dense in
  $\mathcal L_c\mathcal L_cX$. An application of Hahn-Banach theorem
  gives a non-zero continuous linear functional on
  $\mathcal L_c\mathcal L_cX$ that vanishes on
  $j(X)$. Theorem~\ref{LcX-refl} yields that the dual of
  $\mathcal L_c\mathcal L_cX$ is $\mathcal L_cX$. Hence, this
  separating functional is an element of $\mathcal L_cX$; denote it
  by~$f$. Since $f$ vanishes on $j(X)$ implies
  $f(x)=\bigl(j(x)\bigr)(f)=0$ for every $x\in X$, we obtain the
  contradiction $f=0$.
\end{proof}

%

\section{Order convergence in vector lattices}
\label{sec:oconv}

There are several convergences that are critically important in the
theory of vector lattices, yet they are generally non-topological.
These include order convergence, relative uniform convergence, and
unbounded order convergence. While non-topological, these convergences
fit in the framework of the theory of convergence structures. In this
section, we present a survey of various properties of order
convergence from the point of view of convergence
structures. Throughout, $X$ will stand for a vector lattice and we
only consider real scalars; most of the results extend easily to
complex vector lattices. We refer the reader to~\cite{Aliprantis:06}
for background and terminology of the theory of vector lattices and
to~\cite{Abramovich:05} for a discussion on definitions of order
convergence.

\begin{definition}
  \label{def:oconv}
  A net $(x_\alpha)_{\alpha\in\Lambda}$ in a vector lattice $X$ is said to
  \term{converge in order} to some vector $x$ in $X$ if there exists a net
  $(u_\gamma)_{\gamma\in\Gamma}$ such that $u_\gamma\downarrow 0$ and
  for every $\gamma_0\in\Gamma$ there exists an $\alpha_0\in\Lambda$
  such that $\abs{x_\alpha-x}\le u_\gamma$ for all
  $\alpha\ge\alpha_0$. We write $x_\alpha\goeso x$.  Note that the net
  $(u_\gamma)_{\gamma\in\Gamma}$ may be indexed by a different index
  set than the original net. We say that
  $(u_\gamma)_{\gamma\in\Gamma}$ is a \term{dominating} net.
\end{definition}

It is easy to see that order convergence satisfies (N1)--(N3), hence
forms a convergence structure. We will denote this convergence space by $(X,\goeso)$.

Definition~\ref{def:oconv} may be restated as follows:
$x_\alpha\goeso x$ if there exists a net $(u_\gamma)$ such that
$u_\gamma\downarrow 0$ and for every~$\gamma$, the order interval
$[x-u_\gamma,x+u_\gamma]$ contains a tail of $(x_\alpha)$. It is easy
to see that a single net $(u_\gamma)$
may be replaced with two nets
which ``control'' the difference from above and below:

\begin{proposition}\label{oconv-two-nets}
  $x_\alpha\goeso x$ if there exist two nets $(a_\gamma)$ and
  $(b_\gamma)$ such that $a_\gamma\uparrow x$, $b_\gamma\downarrow x$,
  and for every~$\gamma$, the order interval $[a_\gamma,b_\gamma]$
  contains a tail of $(x_\alpha)$.
\end{proposition}

This yields the following useful reformulation of the definition of
order convergence:

\begin{proposition}\label{oconv-nested-int}
  $x_\alpha\goeso x$ iff there is a nested
decreasing net of order intervals such that the intersection of all
these intervals is $\{x\}$ and each of the intervals contains a tail
of $(x_\alpha)$.
\end{proposition}

Proposition~\ref{oconv-nested-int} allows one to extend the definition
of order convergence from vector lattices to partially ordered sets;
cf.\ \cite[p.~171]{Schechter:97}.
While some results of this section
remain valid for partially ordered sets and ordered vector spaces, we
will focus on vector lattices.

It is a standard easy fact that order limits are unique. In the
language of convergence structures this precisely means that order
convergence structure is Hausdorff. In vector lattice theory, a
function $f\colon X\to Y$ between two vector lattices is said to be
\term{order continuous} if $x_\alpha\goeso x$ implies
$f(x_\alpha)\goeso f(x)$; this means that $f$ is 
continuous with respect to order convergence structures.

We will use the fact that if $(x_\alpha)$ is an increasing net then
$x_\alpha\goeso x$ iff $x_\alpha\uparrow x$; the latter means
$x=\sup x_\alpha$. Similarly, for a decreasing net we have
$x_\alpha\goeso x$ iff $x_\alpha\downarrow x$ iff $x=\inf x_\alpha$.

Is order convergence linear? It is easy to see from the definition
that this convergence is translation invariant in the sense that
$x_\alpha\goeso x$ implies $x_\alpha+a\goeso x+a$. This allows one to
reduce many questions about order convergence to order convergence at
zero. Moreover, it is a standard fact that addition is jointly
continuous: if $(x_\alpha)_{\alpha\in A}$ converges in order to $x$
and $(y_\beta)_{\beta\in B}$ converges in order to~$y$, the net
$(x_\alpha+y_\beta)$ indexed by $A\times B$ converges in order to
$x+y$. It is also easy to see that $x_\alpha\goeso x$ implies
$\lambda x_\alpha\goeso \lambda x$ for every $\lambda\in\mathbb
R$. The only outstanding issue is the joint continuity of scalar
multiplication.

\begin{proposition}
  A vector lattice is Archimedean iff its order convergence structure
  is linear.
\end{proposition}

\begin{proof}  
  By the preceding paragraph, we know that order convergence is linear
  iff scalar multiplication is jointly continuous. In the latter case
  $\frac1nu\goeso 0$ for every $u\in X_+$. This means
  $\frac1nu\downarrow 0$, so $X$ is Archimedean.

  Now suppose that $X$ is Archimedean. Let
  $\lambda_\alpha\to\lambda$ in $\mathbb R$ and $x_\beta\goeso x$
  in~$X$; we need to show
  $\lambda_\alpha x_\beta\goeso\lambda x$. By Remark~\ref{same-ind}
  we may assume the two nets have the same index
  set. Passing to a tail, we may also assume $(\lambda_\alpha)$ is
  bounded; i.e., there is some $M>0$ where
  $\abs{\lambda_\alpha}\le M$ for all~$\alpha$. Then
  \begin{displaymath}
    \abs{\lambda_\alpha x_\alpha-\lambda x}
    \le M\abs{x_\alpha-x}+\abs{\lambda_\alpha-\lambda}\abs{x}.
  \end{displaymath}
  Clearly, $M\abs{x_\alpha-x}\goeso 0$. Let
  $u_n=\frac1n\abs{x}$. Since $X$ is Archimedean, $u_n\downarrow 0$.
  For every~$n$, there exists $\alpha_0$ such that
  $\abs{\lambda_\alpha-\lambda}<\frac1n$ for all $\alpha\ge\alpha_0$;
  it follows that
  \begin{math}
    \abs{\lambda_\alpha-\lambda}\abs{x}\le u_n
  \end{math}
  and, therefore,
  \begin{math}
    \abs{\lambda_\alpha-\lambda}\abs{x}\goeso 0.
  \end{math}
  This yields $\lambda_\alpha x_\alpha\goeso\lambda x$.
\end{proof}

The preceding proposition somewhat justifies the importance of the
Archimedean Property and explains why most of the literature on vector
lattices focuses on Archimedean spaces.

\smallskip

It is a standard fact from the theory of vector lattices that lattice
operations $x\vee y$, $x\wedge y$, $\abs{x}$, $x^+$, and $x^-$ are 
order continuous. Indeed, the map $x\to\abs{x}$ is order continuous because
$\bigabs{\abs{x_\alpha}-\abs{x}}\le\abs{x_\alpha-x}$. Similar
arguments show that $x^+$ and $x^-$ are order continuous. It now
follows from the fact that $x\vee y=x+(y-x)^+$ that the map
$(x,y)\mapsto x\vee y$
is jointly continuous. Finally, $x\wedge y=-\bigl((-x)\vee(-y)\bigr)$
yields the joint continuity of the map $(x,y)\mapsto x\wedge y$.

Recall that a subset $A$ of a vector lattice is said to be \term{order
  bounded} when it is contained in an order interval. These sets
correspond to the bounded sets of $(X,\goeso)$ when order convergence
is linear.

\begin{proposition}\label{o-bdd}
  A subset $B$ in an Archimedean vector lattice is order bounded iff
  it is bounded in the order convergence structure.
\end{proposition}

\begin{proof}
  Assume $A$ is order bounded: $A\subseteq[-u,u]$ for some
  $u\in X_+$. Let $(r_\alpha)_{\alpha\in\Lambda}$ be a net in
  $\mathbb R$ with $r_\alpha\to 0$. We need to show that
  $(r_\alpha a)$ converges in order to zero when viewed as a net over
  $\Lambda\times A$ directed by the first component. This follows from
  the fact that $\abs{r_\alpha a}\le\abs{r_\alpha}u\goeso 0$.

  Now suppose $A$ is bounded in $(X,\goeso)$. Then 
  the net $(\frac1n a)$ indexed by $\mathbb N\times A$ converges in
  order to zero. In particular, this net has an order bounded tail. It
  follows that there exists $u\in X_+$ and $n_0\in\mathbb N$ such that
  $\frac{1}{n_0}a\in[-u,u]$ for all $a\in A$. Then
  $A\subseteq[-n_0u,n_0u]$. 
\end{proof}

It is easy to see that every order convergent net has an order bounded tail.
In the language of convergence structures this translates to the following.

\begin{corollary}\label{o-loc-bdd}
  If $X$ is an Archimedean vector lattice then the convergence space
  $(X,\goeso)$ is locally bounded.
\end{corollary}

\begin{example}
  Let $X=\mathbb R^2$ with the lexicographic order. Note that $X$ is
  non-Archimedean and its order convergence structure is not linear.
  Let $A$ be the $y$-axis. Clearly, $A$ is order bounded as
  $A\subseteq[-u,u]$ for $u=(1,0)$. However, $A$ is not bounded in the
  natural convergence structure of $\mathbb R^2$. This example
  illustrates that in the absence of linearity, the concept of
  order boundedness may be ``ill behaved''.
\end{example}

\subsection*{Sublattices.} Let $Y$ be a sublattice of an Archimedean
vector lattice~$X$. In particular, $Y$ is a vector lattice in its own
right. It is well known that order convergence in $Y$ need not
agree with the order convergence inherited from~$X$. In other words, order
convergence generally depends on the ambient space. For example, the
unit vector basis $(e_n)$ in $c_0$ fails to converge in order (it is
not even order bounded) but, viewed as a sequence in~$\ell_\infty$,
it converges in order to zero. It may also happen that
$x_\alpha\goeso 0$ in $Y$ but not in~$X$; see, e.g., Example~2.6
in~\cite{Gao:17}. In particular, a sublattice need not be a subspace
in the sense of convergence structures.

$Y$ is said to be a \term{regular sublattice} of $X$ if
$y_\alpha\goeso 0$ in $Y$ implies $y_\alpha\goeso 0$
in~$X$. Equivalently, the inclusion map is order continuous.

A sublattice $T$ is said to be \term{majorizing} in $X$ if for every
$x\in X_+$ there exists $y\in Y_+$ such that $x\le y$.  We say that
$Y$ is \term{order dense} in $X$ if for every non-zero $x\in X_+$
there exists $y\in Y_+$ such that $0<y\le x$. Theorem~1.23
in~\cite{Aliprantis:03} asserts that every order dense sublattice is
regular. In particular, every Archimedean vector lattice $X$ is order
dense and, therefore, regular in its order (or Dedekind) completion~$X^\delta$.

One may ask whether being order dense is the same as being dense with
respect to the order convergence structure; that is, whether order
dense is equivalent to every $x\in X$ is an order limit of some net
in~$Y$.  It is shown in Theorem~1.27 of~\cite{Aliprantis:03} that $Y$
is order dense iff $a=\sup[0,a]\cap Y$ for every $a\in X_+$ where the
supremum is evaluated in~$X$. It follows that if $Y$ is order dense
then it is dense with respect to the order convergence
structure. Example~2.6 in~\cite{Gao:17} shows that the converse is
generally false. Thus, there is an unfortunate clash of terminologies:
a sublattice which is dense in the sense of order convergence
structure need not be order dense. However, Corollary~2.13
of~\cite{Gao:17} shows that the two concepts of density agree for
regular sublattices. In particular, every Archimedean vector lattice
$X$ is dense in $X^\delta$ in the sense of the order convergence
structure.

Theorem~2.8 of~\cite{Gao:17} asserts that if $Y$ is order dense and
majorizing, and $(y_\alpha)$ is a net in $Y$ then $y_\alpha\goeso 0$
in $Y$ iff $y_\alpha\goeso 0$ in~$X$. That is, $Y$ is a convergence
subspace of~$X$. A special case of this is the following fact that was
established in~\cite{Abramovich:05}:

\begin{theorem}\label{AS}
  Every Archimedean vector lattice $X$ is a convergence subspace in
  its order completion~$X^\delta$; i.e., $x_\alpha\goeso 0$ in $X$ iff
  $x_\alpha\goeso 0$ in $X^\delta$ for every net $(x_\alpha)$ in~$X$.
\end{theorem}

It is also worth mentioning that Corollary~2.12 in~\cite{Gao:17} says
that if $Y$ is a regular sublattice of $X$ then the order convergences
of $X$ and of $Y$ agree for order bounded nets.

\subsection*{Order completeness}

Recall that a linear convergence structure is complete when every
Cauchy net is convergent. A vector lattice $X$ is said to be \term{order (or
Dedekind) complete} if every bounded above increasing net has a
supremum. We are going to show that an Archimedean vector lattice is
order complete iff its order convergence structure is complete; cf.\
Proposition~2.3 in~\cite{Gao:17}. 

Let $X$ be an order complete vector lattice and $(x_\alpha)$ an order
bounded net in~$X$. It is a standard exercise that $x_\alpha\goeso x$ iff 
\begin{math}
     x=\inf_{\alpha}\sup_{\beta\ge\alpha}x_\beta=      
       \sup_{\alpha}\inf_{\beta\ge\alpha}x_\beta.      
\end{math}
It follows that in Proposition~\ref{oconv-two-nets} we can take
$a_\alpha=\inf_{\beta\ge\alpha}x_\beta$ and
$b_\alpha=\sup_{\beta\ge\alpha}x_\beta$. In particular, this means
that the ``control'' nets may be chosen over the same index set as the
original net. Then the dominating net in Definition~\ref{def:oconv}
may also be chosen over the same index set as $(x_\alpha)$. Thus, in
this setting we obtain $x_\alpha\goeso x$ iff there exists a net
$(u_\alpha)$ over the same index where $u_\alpha\downarrow 0$ and
$\abs{x_\alpha-x}\le u_\alpha$ for every~$\alpha$.

A net $(x_\alpha)_{\alpha\in\Lambda}$ in a vector lattice $X$ is said
to be \term{order Cauchy} if it is Cauchy in the order convergence
structure; i.e., if the net $(x_\alpha-x_\beta)_{(\alpha,\beta)}$ is
order null. The latter net is indexed by $\Lambda\times\Lambda$ under the
component-wise order.

\begin{lemma}
  Every monotone order bounded net in an Archimedean vector lattice is
  order Cauchy.
\end{lemma}

\begin{proof}
  Suppose that $x_\alpha\uparrow\le u$ in~$X$. Then $z:=\sup x_\alpha$
  exists in~$X^\delta$. It follows that the double net
  $(x_\alpha-x_\beta)_{(\alpha,\beta)}$ is order null 
  in $X^\delta$ and, therefore, in~$X$. The case of a decreasing net
  may be handled in a similar way.
\end{proof}

\begin{proposition}
  An Archimedean vector lattice is order complete iff every order
  Cauchy net is order convergent.
\end{proposition}

\begin{proof}
  Let $(x_\alpha)$ be an order Cauchy net in an order complete vector
  lattice; we will show that it is order convergent. It is easy to
  see that $(x_\alpha)$ has an order bounded tail, so 
  we may assume that $(x_\alpha)$ is order bounded. Put
  \begin{displaymath}
    x=\sup a_\alpha\mbox{ where }
    a_\alpha=\inf_{\beta\ge\alpha}x_\beta\mbox{, and let }
    y=\inf b_\alpha\mbox{ where }
    b_\alpha=\sup_{\beta\ge\alpha}x_\beta.
  \end{displaymath}
  Then $a_\alpha\le x\le y\le b_\alpha$ for every~$\alpha$. By
  the preceding argument, it suffices to show that $x=y$. Let
  $(v_{\alpha,\beta})$ be a net such that
  $v_{\alpha,\beta}\downarrow 0$ and
  $\abs{x_\alpha-x_\beta}\le v_{\alpha,\beta}$. Fix a pair
  $(\alpha_0,\beta_0)$. Let $\alpha$ be such that $\alpha\ge\alpha_0$
  and $\alpha\ge\beta_0$. For every $\beta$ with $\beta\ge\alpha$, we
  have $(\alpha,\beta)\ge(\alpha_0,\beta_0)$, so that
  $\abs{x_\alpha-x_\beta}\le v_{\alpha_0,\beta_0}$. It follows that
  \begin{math}
    x_\beta\in[x_\alpha-v_{\alpha_0,\beta_0},x_\alpha+v_{\alpha_0,\beta_0}],
  \end{math}
  which yields
  \begin{math}
    a_\alpha,b_\alpha\in[x_\alpha-v_{\alpha_0,\beta_0},x_\alpha+v_{\alpha_0,\beta_0}].
  \end{math}
  We thus obtain $0\le b_\alpha-a_\alpha\le 2 v_{\alpha_0,\beta_0}$ and
  so $0\le y-x\le 2v_{\alpha_0,\beta_0}$. If follows that $x-y=0$.

  Conversely, suppose that every order Cauchy net in $X$ is order
  convergent. Let $0\le x_\alpha\!\!\uparrow\le u$. By the lemma,
  $(x_\alpha)$ is order Cauchy, hence order convergent; it follows
  that $\sup x_\alpha$ exists.
\end{proof}

Recall that the most common construction of the order completion of an
Archimedean vector lattice is analogous to the construction of
$\mathbb R$ using Dedekind cuts; see, for example,
\cite{Vulikh:67}. Also recall the alternate construction of
$\mathbb R$ using equivalence classes of Cauchy sequences of
rationals.

The latter construction has an analogue for Archimedean vector
lattices. Given an Archimedean vector lattice~$X$, we may view the
order convergence structure of its order completion $X^\delta$ as a
\term{Cauchy completion} of the order convergence structure on $X$ in
the following way: $X$ is a dense subspace of $X^\delta$ in the sense
of convergence structures, and the elements of $X^\delta$ are limits
of order Cauchy nets in~$X$. To see this, note that for every
$x\in X^\delta$ there exists a net $(x_\alpha)$ in $X$ with
$x_\alpha\goeso x$ in~$X^\delta$. It follows that $(x_\alpha)$ is
order Cauchy in $X^\delta$ and, therefore, in $X$ by
Proposition~\ref{AS}. Conversely, every order Cauchy net in $X$
remains order Cauchy in $X^\delta$ by Proposition~\ref{AS}; the
previous proposition now yields that it converges in order in~$X^\delta$.



\subsection*{Filter considerations.}

Proposition~\ref{oconv-nested-int} yields a filter version of order
convergence: a filter $\mathcal F$ on a vector lattice $X$ converges
in order to $x$ (written $\mathcal F\goeso x$) if $\mathcal F$
contains a nested decreasing family of order intervals whose
intersection is~$\{x\}$; see~\cite[p.~54]{Schaefer:74}
or~\cite{Anguelov:05}.  It can now be easily verified that this filter
convergence structure corresponds exactly to the net convergence
structure in Definition~\ref{def:oconv}. This allows one to use the
full power of the theory of filter convergence structures when dealing
with order convergence.

In particular, we can discuss the concepts of local convexity and
regularity from the theory of filter convergence structures. A
convergence space is \term{regular} when $\mathcal F\to x$ implies
that the filter generated by the closures of the sets in $\mathcal F$
still converges to~$x$; this is not the same as a regular sublattice.
A convergence vector space is
\term{locally convex} if $\mathcal F\to x$ implies the filter
generated by the convex hulls of the sets in $\mathcal F$ still
converges to~$x$.
Now let $X$ be a vector lattice equipped with its (filter) order
convergence structure. Since order intervals are order closed and
convex, we immediately obtain the following results. 

\begin{proposition}
	The order convergence structure is regular and locally convex.
\end{proposition}

\section{Relative uniform convergence and order topology}
\label{sec:uconv-otop}

We will show in this section that the Mackey modification of order
convergence is relative uniform convergence. We then discuss the topological
modification of order convergence and characterize the spaces where
order convergence is topological.

\smallskip

\subsection*{Relative uniform convergence.}
Let $X$ be a vector lattice. We say that a net $(x_\alpha)$ converges
to some $x\in X$ \term{(relatively) uniformly} and write
$x_\alpha\goesu x$ if there exists $e\in X_+$ such that for every
$\varepsilon>0$ there exists $\alpha_0$ such that
$\abs{x_\alpha-x}\le\varepsilon e$ for all $\alpha\ge\alpha_0$.

Let $e\in X_+$. For $x\in X$, put
\begin{displaymath}
    \norm{x}_e=\inf\bigl\{\lambda\ge 0\mid\abs{x}\le\lambda e\bigr\}.
\end{displaymath}
It is clear that $\norm{x}_e<\infty$ iff $x$ is in
the principal ideal $I_e$ of~$e$. It is also easy to see that
$\norm{\cdot}_e$ is a lattice seminorm on~$I_e$; it is a norm when $X$
is Archimedean. Then $x_\alpha\goesu x$ iff
$\norm{x_\alpha-x}_e\to 0$ for some $e\in X_+$.

If $X$ is an Archimedean vector lattice then relative uniform convergence
is a linear Hausdorff convergence structure on~$X$.  This may be
verified directly. It also follows immediately from the next result:

\begin{theorem}\label{u-Mack-o}
  Let $X$ be an Archimedean vector lattice. Then its relative uniform
  convergence structure is the Mackey modification of its order
  convergence structure.
\end{theorem}

\begin{proof}
  Recall that by Proposition~\ref{o-bdd}, order bounded sets in $X$ are
  the sets that are bounded in the order convergence structure. The
  result now follows easily from the fact that a set is order bounded
  iff it is contained in $[-e,e]$ for some $e\in X_+$.
\end{proof}

The sequential variant of this theorem was proved in~\cite{VanderWalt:16}.

\begin{corollary}
  Let $\mathcal F$ be a filter in an  Archimedean vector lattice. Then
  $\mathcal F\goesu 0$ iff there exists $e\in X_+$ such that
  $[-\frac1ne,\frac1ne]\in\mathcal F$ for every $n\in\mathbb N$.
\end{corollary}

The following result characterizes order boundedness of an operator as
a form of continuity. This was proved in Theorem~5.1 of
\cite{Taylor:20}, but now it follows immediately from
Theorem~\ref{u-Mack-o} and Proposition~\ref{Mack-bdd}.

\begin{theorem}
  Let $T\colon X\to Y$ be a linear operator between Archimedean vector
  lattices. The following are equivalent:
  \begin{enumerate}
  \item $T$ is order bounded; that is, $T$ maps order bounded sets to order
    bounded sets.
  \item $T$ is relative uniformly continuous: $x_\alpha\goesu x$ implies
    $Tx_\alpha\goesu Tx$.
  \item $T$ is relative uniformly-to-order continuous; that is,
    $x_\alpha\goesu x$ implies $Tx_\alpha\goeso Tx$.
  \end{enumerate}
\end{theorem}

It is also observed in \cite{Taylor:20} that a Banach lattice is order
continuous iff order convergence and relative uniform convergence
structures agree iff order convergence and relative uniform
convergence agree on sequences. Thus, the order convergence structure
on Banach lattice $X$ is Mackey iff $X$ is order continuous.
Combining this with the preceding theorem, we get the following.

\begin{corollary}
  A linear operator from an order continuous Banach lattice to an
  Archimedean vector lattice is order bounded iff it is order
  continuous.
\end{corollary}

In the filter language, it is easy to see that $\mathcal F\goesu 0$
iff there exists $u\in X_+$ such that $\mathcal
N_0[-u,u]\subseteq\mathcal F$. Here, as before, $\mathcal N_0$ stands
for the neighborhood filter of zero in~$\mathbb R$. Let $\mathcal G$
be the filter generated by  $\mathcal
N_0[-u,u]$; it is easy to see that the
order intervals $\bigl[-\frac1n u,\frac1nu\bigr]$ as $n\in\mathbb N$
form a base of~$\mathcal G$. This proves the following:

\begin{proposition}
  Relative uniform convergence of an Archimedean vector lattice is
  first countable.
\end{proposition}

It was also observed in \cite{Taylor:20} that if $Y$ is a relative uniformly
closed sublattice of a relative uniformly complete vector lattice $X$
and $(x_k)$ is a sequence in~$Y$, then $x_k\goesu 0$ in $Y$ iff
$x_k\goesu 0$ in~$X$.

\subsection*{Order Convergence and Topology}
As we mentioned in the beginning of the paper, Ordman observed
in~\cite{Ordman:66} that almost everywhere convergence of measurable
functions is not topological. The argument is very simple; we
reproduce it here for the convenience of the reader:

\begin{example}\label{typewriter}
  Consider the space $L_0[0,1]$ of all Lebesgue measurable functions
  on $[0,1]$ equipped with convergence almost everywhere.  Let $s_0=0$
  and $s_n=\sum_{k=1}^n\frac1k$ for $n\in\mathbb N$, then put
  $A_n=[s_{n-1},s_n]\mod 1$. Then $A_n\subseteq[0,1]$ and the Lebesgue
  measure of $A_n$ is~$\frac1n$. The sequence $(x_n)$ defined by
  $x_n=\one_{A_n}$ is called the \term{typewriter sequence}. It is
  easy to see that every subsequence of $(x_n)$ has a further
  subsequence that converges to zero. If the convergence were
  topological, we would have $x_n\to 0$, but the latter is false.
\end{example}

For sequences in $L_0[0,1]$, order convergence agrees with convergence
almost everywhere, hence the typewriter sequence from the example
shows that order convergence is generally not topological. We will
show that order convergence may be very far from being topological; it
is topological iff the space is finite-dimensional. Moreover, its
topological modification is the order topology, which may be very
coarse: it need not be linear or Hausdorff.

In vector lattice theory, a subset $A$ of a vector lattice is said to
be \term{order closed} if $x_\alpha\goeso x$ and
$\{x_\alpha\}\subseteq A$ imply $x\in A$; this agrees with the
definition of closed sets in the order convergence structure. The
collection of all order closed sets forms a topology;
following~\cite[p.~80]{Luxemburg:71}, we call it the \term{order
topology}. In the language of convergence structures, this is
precisely the topological modification of the order convergence
structure; we will denote this topology by~$\tau_o$. By
Proposition~\ref{tmod-finest}, $\tau_o$ is the finest topology on $X$
whose convergence is weaker than order convergence.

It is easy to see that the order topology is translation
invariant, and it follows immediately from the definition that points are
closed. By Theorem~5.1 on pp.~34--35 of \cite{Kelley:76}, a
topological vector space is Hausdorff iff points are
closed. Therefore, if $\tau_o$ is linear, it is automatically
Hausdorff. The following example shows the order topology need
not be Hausdorff or linear; cf.\
\cite[p.~146]{Vladimirov:69}. 

\begin{example}\label{otop-nonHaus}
    \emph{Order topology on $C[0,1]$ is not Hausdorff and not linear.}
  Suppose it is Hausdorff. Let $U$ and $V$ be open neighbourhoods of
  $0$ and of~$\one$, respectively, with $U\cap V=\varnothing$. Let
  $(t_k)$ be an enumeration of the rational numbers in $(0,1)$. For
  each $k,n\in\mathbb N$, let $x_{k,n}$ be the continuous function
  such that $x_{k,n}(t_k)=1$, $x_{n,k}$ vanishes outside of
  $\bigl[t_k-\frac1n,t_k+\frac1n\bigr]$, and is linear on
  $\bigl[t_k-\frac1n,t_k\bigr]$ and on
  $\bigl[t_k,t_k+\frac1n\bigr]$. For each fixed~$k$, we clearly have
  $x_{k,n}\downarrow 0$ as $n\to\infty$. It follows that
  $x_{k,n}\goeso 0$ and, therefore, $x_{k,n}\xrightarrow{\tau_o}x$ as
  $n\to\infty$. Choose $n_1$ such that $y_1:=x_{1,n_1}\in U$. Next we observe
  that $y_1\vee x_{2,n}\downarrow y_1$. It follows from $y_1\in U$
  that $y_2:=y_1\vee x_{2,n_2}\in U$ for some~$n_2$. Iterating this
  process, we construct a sequences $(n_k)$ in $\mathbb N$ and $(y_k)$
  in $U$ such that $y_{k}=y_{k-1}\vee x_{k,n_k}$ for every $k>1$. It
  follows that $y_k\uparrow$ and $y_k(t_i)=1$ as $i=1,\dots,k$. This
  yields $y_k\uparrow\one$. Therefore, there exists $k_0$ such that
  $y_{k_0}\in V$, which contradicts $U$ and $V$ being disjoint.

  So the order topology on $C[0,1]$ can fail to be Hausdorff and, therefore,
  is not generally a linear topology.
\end{example}

While the preceding example shows that the order topology may be rather
poor, there are situations where it is ``nice''.
Let $X$ be Banach lattice. It is well known that every norm convergent
sequence in $X$ has a subsequence that converges in order to the same
limit. It follows that every order closed set in $X$ is norm
closed: so the norm topology on $X$ is stronger than its order
topology. If, in addition, $X$ is order continuous, a
set is norm closed iff it is order closed. Thus, the topological
modification of order convergence
agrees with the norm topology in order continuous
Banach lattices.

We now provide a useful characterization of order neighbourhoods,
which are sometimes called \emph{net-catching sets} in the literature.

\begin{proposition}
  Let $U$ be a subset of a vector lattice $X$ and $x\in X$. The
  following are equivalent:
  \begin{enumerate}
  \item\label{o-nbhd-nbhd}  $U$ is an order neighborhood of~$x$;
  \item\label{o-nbhd-net-catch} if $u_\alpha\downarrow 0$ then there
    exists $\alpha_0$ such that $[x-u_\alpha,x+u_\alpha]\subseteq U$
    for all $\alpha\ge\alpha_0$.
 \end{enumerate}
\end{proposition}

\begin{proof}
 Without loss of generality we may assume $x=0$.

  \eqref{o-nbhd-nbhd}$\Rightarrow$\eqref{o-nbhd-net-catch} Suppose
  that $U$ satisfies \eqref{o-nbhd-nbhd} but fails
  \eqref{o-nbhd-net-catch}. Then there exists a net $(u_\alpha)$ such that
  $u_\alpha\downarrow 0$ but for every~$\alpha$, the interval
  $[-u_\alpha,u_\alpha]$ is not entirely contained in~$U$. We can then
  find $x_\alpha$ in this interval such that $x_\alpha\notin U$. This
  results in a net $(x_\alpha)$ such that $x_\alpha\goeso 0$, yet no
  $x_\alpha$ is in~$U$; a contradiction.
  
  \eqref{o-nbhd-net-catch}$\Rightarrow$\eqref{o-nbhd-nbhd} Suppose
  that $(x_\alpha)_{\alpha\in\Lambda}\to 0$. Let
  $(u_\gamma)_{\gamma\in\Gamma}$ be a dominating net as in
  Definition~\ref{def:oconv}. By assumption, there exists $\gamma_0$
  such that $[-u_{\gamma_0},u_{\gamma_0}]\subseteq U$. On the other
  hand, a tail of $(x_\alpha)$ is contained in
  $[-u_{\gamma_0},u_{\gamma_0}]$, and, therefore in~$U$.
\end{proof}

It was shown in~\cite{Dabboorasad:20} that order convergence is
topological precisely in finite-dimensional
spaces. We will now provide a short proof of this fact.

\begin{lemma}\label{top-to-o-conv-sunit}
  Let $\tau$ be a linear topology on a vector lattice~$X$. If
  $\tau$-convergence is stronger than order convergence then $X$ is
  Archimedean, has a strong unit~$e$, and $x_\alpha\goestau 0$ implies
  $\norm{x_\alpha}_e\to 0$.
\end{lemma}

\begin{proof}
  To show that $X$ is Archimedean, suppose
  $0\le x\le\frac{1}{n}u$ for all $n\in\mathbb N$. Since scalar
  multiplication is $\tau$-continuous, we have
  $\frac{1}{n}u\goestau 0$, which yields $\frac{1}{n}u\goeso 0$ and,
  therefore, $x=0$.
  
  Let $\mathcal N_0$ be the filter of zero neighborhoods for~$\tau$. Then
  we have $\mathcal N_0\goestau 0$ and, therefore,
  $\mathcal N_0\goeso 0$. Since order convergence is locally bounded
  by Corollary~\ref{o-loc-bdd}, $\mathcal N_0$ contains an order
  bounded member. That is, there exists a $U\in\mathcal N_0$ and
  $e\in X_+$ such that $U\subseteq[-e,e]$. It follows that $e$ is a
  strong unit.
  
  To prove the last claim, suppose that $x_\alpha\goestau 0$ for some
  net $(x_\alpha)$. Fix $\varepsilon>0$ and let $U$ be as above. Then
  $\varepsilon U\in\mathcal N_0$ and there exists $\alpha_0$ such that
  $x_\alpha\in\varepsilon U$ whenever $\alpha\ge\alpha_0$. It follows
  that $\abs{x_\alpha}\le\varepsilon e$; i.e.,
  $\norm{x_\alpha}_e\le\varepsilon$. 
\end{proof}

\begin{theorem}\label{oconv-lin-top}
  Order convergence on a vector lattice $X$ arises from a linear
  topology iff $X$ is finite-dimensional and Archimedean. 
\end{theorem}

\begin{proof}
  Let $X$ be a finite-dimensional Archimedean vector lattice. Then $X$
  is lattice-isomorphic to $\mathbb R^n$ for some~$n$. So, without
  loss of generality, we assume $X=\mathbb R^n$. Then order
  convergence agrees with the coordinate-wise convergence, which is
  topological.

  Suppose that there is a linear topology $\tau$ on $X$ with
  $x_\alpha\goestau 0$ iff $x_\alpha\goeso 0$ for every net
  $(x_\alpha)$ in~$X$. By Theorem~\ref{top-to-o-conv-sunit} $X$ is
  Archimedean, has a strong unit~$e$, and $x_\alpha\goestau 0$ implies
  $\norm{x_\alpha}_e\to 0$. For the sake of contradiction, assume that
  $X$ is infinite-dimensional. Then $X$ contains an infinite disjoint
  sequence of non-zero vectors, say, $(x_n)$. Without loss of
  generality, we may assume $\norm{x_n}_e=1$ for each~$n$. Viewed as
  an order bounded disjoint sequence in~$X^\delta$, $(x_n)$ is order
  null in~$X^\delta$. Then $x_n\goeso 0$ in~$X$ by
  Theorem~\ref{AS}. Our assumption gives $x_n\goestau 0$; hence
  $\norm{x_n}_e\to 0$ by Lemma~\ref{top-to-o-conv-sunit}, a
  contradiction.
\end{proof}

\begin{corollary}\label{order-topol}
  Order convergence on an Archimedean vector lattice $X$ is
  topological iff $X$ is finite-dimensional. 
\end{corollary}

\begin{proof}
  Suppose that order convergence on $X$ comes from a
  topology~$\tau$. Since addition and scalar multiplication are
  jointly continuous with respect to order convergence, it follows
  that $\tau$ is linear. Now apply Theorem~\ref{oconv-lin-top}.
\end{proof}

\begin{example}
  Consider $X=\mathbb R^2$ with the lexicographic order. Note that $X$ fails
  the Archimedean Property. The sequence $(\frac{1}{n},0)$ converges
  to zero in every Hausdorff linear topology on~$X$, yet it does not
  converge to zero in order.
\end{example}

\begin{proposition}\label{unif-topol}
  Relative uniform convergence on an Archimedean vector lattice $X$ is
  topological iff $X$ has a strong unit.
\end{proposition}

\begin{proof}
  If relative uniform convergence arises from a topology, then $X$ has
  a strong unit by Lemma~\ref{top-to-o-conv-sunit}. Conversely, if $X$
  has a strong unit, say,~$e$, then relative uniform convergence
  agrees with convergence in $\norm{\cdot}_e$-norm.
\end{proof}

\begin{remark}
  It was observed in Proposition~3.1.3 of~\cite{Beattie:02} that a
  convergence vector space is topological iff it is
  pretopological. This implies that ``topological'' may be replaced
  with ``pretopological'' in Corollary~\ref{order-topol} and
  Proposition~\ref{unif-topol}. In particular, in an
  infinite-dimensional Archimedean vector lattice, an arbitrary mixing
  of order (or uniformly) convergent nets may spoil the convergence.
\end{remark}

\bigskip

In the last few years, several variants of \emph{unbounded
  convergences} came to prominence in vector lattice theory. These
provide more examples of non-topological convergence structures. We
will discuss unbounded convergences from the perspective of the theory 
of convergence structures in a separate paper.

\bigskip

{\bf Acknowledgements.}
The authors would like to thank M.~Taylor for valuable discussions.

\end{document}